\documentclass[10pt,twoside]{article}
\usepackage{mathrsfs}
\usepackage{amsmath}
\usepackage{amssymb}
\usepackage{fancyhdr}
\usepackage{latexsym}
\usepackage{bbding}
\usepackage{mathrsfs}
\usepackage{wasysym}
\usepackage{multicol,graphics}

\newtheorem{theorem}{Theorem}[section]
\newtheorem{lemma}[theorem]{Lemma}

\newtheorem{definition}[theorem]{Definition}

\newtheorem{proposition}[theorem]{Proposition}
\numberwithin{equation}{section}
\newenvironment{proof}[1][Proof]{\noindent\textbf{#1.} }{\hfill $\Box$}
\allowdisplaybreaks

 \makeatletter
\begin{document}
\title{{Well-posedness and decay for the dissipative system modeling electro-hydrodynamics in negative Besov spaces}}
\author{Jihong Zhao$^{\text{1}}$\footnote{Email addresses: jihzhao@163.com, zhaojih@nwsuaf.edu.cn (J. Zhao);
liuqao2005@163.com (Q. Liu).},
\ \ Qiao Liu$^{\text{2}}$\\
[0.2cm] {\small $^{\text{1}}$ College of Science, Northwest A\&F
University,  Yangling, Shaanxi 712100, China}\\
[0.2cm] {\small $^{\text{2}}$ Department of Mathematics, Hunan
Normal University, Changsha, Hunan 410081,  China}}

\date{}
\maketitle

\begin{abstract}
In \cite{GW12} (Y. Guo, Y. Wang, Decay of dissipative equations and
negative Sobolev spaces,  Commun. Partial Differ. Equ. 37 (2012)
2165--2208), Y. Guo and Y. Wang developed a general new energy method for proving the optimal time decay rates of the solutions to dissipative equations. In this paper, we generalize this method in the framework of homogeneous Besov spaces. Moreover, we apply this method to a model arising from
electro-hydrodynamics, which is a strongly coupled system of the
Navier-Stokes equations and the Poisson-Nernst-Planck equations
through charge transport and external forcing terms. We show that the negative Besov norms are preserved along time evolution,  and obtain the optimal time decay rates of the higher-order spatial derivatives of solutions by the Fourier splitting approach and the interpolation techniques.


\end{abstract}

\section{Introduction}
In \cite{GW12}, Y. Guo and Y. Wang developed a new energy approach to establish the optimal time decay rates of the solutions to the Cauchy problem of the heat equation:
\begin{equation}\label{eq1.1}
\begin{cases}
\partial_{t}u-\Delta u=0,\ \ & x\in\mathbb{R}^{3},\ t>0,\\
u(x,0)=u_{0}(x),\ \ & x\in\mathbb{R}^{3}.
\end{cases}
\end{equation}
They proved the following result:
\begin{theorem}\label{th1.1}
If $u_{0}\in H^{N}(\mathbb{R}^{3})\cap \dot{H}^{-s}(\mathbb{R}^{3})$ with $N\ge 0$ be an integer and $s\ge0$ be a real number, then for any
real number $\ell\in[-s,N]$, there exists a constant $C_{0}$ such that
\begin{equation}\label{eq1.2}
  \|\nabla^{\ell}u(t)\|_{L^2}\le C_{0}(1+t)^{-\frac{\ell+s}{2}}.
\end{equation}
\end{theorem}
Here $H^{s}(\mathbb{R}^{3})$ and $\dot{H}^{s}(\mathbb{R}^{3})$ denote the nonhomogeneous Sobolev space and the homogeneous Sobolev space, respectively.

In this paper, we generalize this new energy approach in the framework of Besov spaces. In order to illustrate this approach, we revisit the heat equation \eqref{eq1.1}.
\begin{theorem}\label{th1.2} Let $N\ge 0$ be an integer and $s\ge0$ be a real number, $1\leq p<\infty$.
If $u_{0}\in \dot{B}^{N}_{p,1}(\mathbb{R}^{3})\cap\dot{B}^{-s}_{p,1}(\mathbb{R}^{3})$, then for any
real number $\ell\in[-s,N]$, there exists a constant $C_{0}$ such that
\begin{equation}\label{eq1.3}
  \|u(t)\|_{\dot{B}^{\ell}_{p,1}}\le C_{0}(1+t)^{-\frac{\ell+s}{2}}.
\end{equation}
\end{theorem}
\begin{proof}
Let $\ell\in[-s,N]$. Applying the dyadic operator $\Delta_{j}$ to the heat equation \eqref{eq1.1}, we see that
\begin{equation*}
  \partial_{t}\Delta_{j}u-\Delta\Delta_{j}u=0,
\end{equation*}
which taking the standard $L^{2}$ inner product with $|\Delta_{j}u|^{p-2}\Delta_{j}u$ leads to
\begin{equation*}
  \frac{1}{p}\frac{d}{dt}\|\Delta_{j}u\|_{L^{p}}^{p}-\int_{\mathbb{R}^{3}}\Delta\Delta_{j}u|\Delta_{j}u|^{p-2}\Delta_{j}udx=0.
\end{equation*}
Thanks to \cite{D01}, there exists a constant $\kappa$ such that
\begin{equation*}
  -\int_{\mathbb{R}^{3}}\Delta\Delta_{j}u|\Delta_{j}u|^{p-2}\Delta_{j}udx\geq \kappa2^{2j}\|\Delta_{j}u\|_{L^{p}}^{p}.
\end{equation*}
Thus, we obtain
\begin{equation*}
   \frac{d}{dt}\|\Delta_{j}u\|_{L^{p}}+\kappa2^{2j}\|\Delta_{j}u\|_{L^{p}}\leq 0.
\end{equation*}
Multiplying the above inequality  by $2^{j\ell}$, then taking $l^{1}$ norm to the resultant yields that
\begin{equation}\label{eq1.4}
  \frac{d}{dt}\|u\|_{\dot{B}^{\ell}_{p,1}}+\kappa\|u\|_{\dot{B}^{\ell+2}_{p,1}}\leq 0.
\end{equation}
Integrating the above in time, we obtain
\begin{equation}\label{eq1.5}
  \|u\|_{\dot{B}^{\ell}_{p,1}}\leq \|u_{0}\|_{\dot{B}^{\ell}_{p,1}}.
\end{equation}
This implies that inequality \eqref{eq1.3} holds  in particular  with $\ell=-s$. Now for $-s<\ell\leq N$, we use the interpolation relation, see Lemma \ref{le5.2} below,  to get
\begin{equation*}
  \|u\|_{\dot{B}^{\ell}_{p,1}}\leq \|u\|_{\dot{B}^{-s}_{p,1}}^{\frac{2}{\ell+s+2}}\|u\|_{\dot{B}^{\ell+2}_{p,1}}^{\frac{\ell+s}{\ell+s+2}},
\end{equation*}
which combining \eqref{eq1.5} implies that
\begin{equation}\label{eq1.6}
  \|u\|_{\dot{B}^{\ell+2}_{p,1}}\geq \|u_{0}\|_{\dot{B}^{-s}_{p,1}}^{-\frac{2}{\ell+s}}\|u\|_{\dot{B}^{\ell}_{p,1}}^{1+\frac{2}{\ell+s}}.
\end{equation}
Plugging \eqref{eq1.6} into \eqref{eq1.4}, we conclude that there exists a constant $C_{0}$ such that
\begin{equation*}
  \frac{d}{dt}\|u\|_{\dot{B}^{\ell}_{p,1}}+C_{0}\|u\|_{\dot{B}^{\ell}_{p,1}}^{1+\frac{2}{\ell+s}}\leq 0.
\end{equation*}
Solving this inequality implies that
\begin{equation*}
  \|u\|_{\dot{B}^{\ell}_{p,1}}\leq \big(\|u_{0}\|_{\dot{B}^{\ell}_{p,1}}^{-\frac{2}{\ell+s}}+\frac{2C_{0}t}{\ell+s}\big)^{-\frac{\ell+s}{2}}\leq C_{0}(1+t)^{-\frac{\ell+s}{2}}.
\end{equation*}
We complete the proof of Theorem \ref{th1.2}.
\end{proof}

\noindent\textbf{Organization of the paper} In Section 2, we make some preliminary preparations. In Section 3, we state our main results. Section 4 is devoted to giving the proofs of Theorems \ref{th3.1} and \ref{th3.2}. In the final Appendix,  we first collect some analytic tools used in this paper, then  give a sketched proof of the global existence of solutions with small initial data in Theorem \ref{th3.1}.

\section{Preliminaries}

\subsection{Notations}

In this paper, we shall use the following notations.
\begin{itemize}
\item For two
constants $A$ and $B$, the notation $A\lesssim B$ means that there
is a uniform constant $C$ (always independent of $x, t$), which may vary from line to line, such that $A\le CB$. $A\approx B$ means that $A\lesssim B$
and $B\lesssim A$.
\item For a quasi-Banach space $X$ and for any $0<T\leq\infty$, we use standard notation $L^{p}(0,T; X)$ or $L^{p}_{T}(X)$  for the quasi-Banach space of Bochner measurable functions
$f$ from $(0, T)$ to $X$ endowed with the norm
\begin{equation*}
\|f\|_{L^{p}_{T}(X)}:=
\begin{cases}
  (\int_{0}^{T}\|f(\cdot,t)\|_{X}^{p}dt)^{\frac{1}{p}}\ \ \ &\text{for}\ \ \ 1\leq p<\infty,\\
  \sup_{0\leq t\leq T}\|f(\cdot,t)\|_{X}\ \ \ &\text{for}\ \ \ p=\infty.
\end{cases}
\end{equation*}
In particular, if $T=\infty$, we use $\|f\|_{L^{p}_{t}(X)}$ instead of $\|f\|_{L^{p}_{\infty}(X)}$.

\item We shall denote by $(f|g)$ the
$L^{2}(\mathbb{R}^{3})$ inner product of two functions $f$ and $g$.

\item $(d_{j})_{j\in\mathbb{Z}}$ will be a generic element of
$l^{1}(\mathbb{Z})$ so that $d_{j}\ge0$ and
$\sum_{j\in\mathbb{Z}}d_{j}=1$.

\item We say that a vector $u=(u^{1}, u^{2}, u^{3})$ belongs to a
function space $X$ if $u^{j}\in X$ holds for every $j=1,2,3$ and we
put $\|u\|_{X}:=\max_{1\leq j
\leq3}\|u^{j}\|_{X}$.

\item Given two quasi-Banach spaces $X$ and $Y$, the
product of these two spaces $X\times Y$ will be equipped with the
usual norm $\|(u,v)\|_{X\times
Y}:=\|u\|_{X}+\|v\|_{Y}$.
\end{itemize}

\subsection{Littlewood-Paley theory and Besov spaces}

Let
$\mathcal{S}(\mathbb{R}^{3})$ be the Schwartz class of rapidly
decreasing function, and $\mathcal{S}'(\mathbb{R}^{3})$ of temperate distributions be the dual set of $\mathcal{S}(\mathbb{R}^{3})$.
Let $\varphi\in\mathcal{S}(\mathbb{R}^{3})$ be a  smooth radial function valued in $[0,1]$ such that  $\varphi$ is supported in the shell $\mathcal{C}=\{\xi\in\mathbb{R}^{3},\ \frac{3}{4}\leq
|\xi|\leq\frac{8}{3}\}$, and
\begin{align*}
  \sum_{j\in\mathbb{Z}}\varphi(2^{-j}\xi)=1, \ \ \ \forall\xi\in\mathbb{R}^{3}\backslash\{0\}.
\end{align*}
Then for any $f\in\mathcal{S}'(\mathbb{R}^{3})$, we define
for all $j\in \mathbb{Z}$,
\begin{align}\label{eq2.1}
  \Delta_{j}f:=\varphi(2^{-j}D)f \ \ \ \text{and}\ \ \
  S_{j}f:=\sum_{k\leq j-1}\Delta_{k}f.
\end{align}
By telescoping the series, we have the following homogeneous Littlewood-Paley decomposition:
\begin{equation*}
  f=\sum_{j\in\mathbb{Z}}\Delta_{j}f \ \ \text{for}\ \
  f\in\mathcal{S}'(\mathbb{R}^{3})/\mathcal{P}(\mathbb{R}^{3}),
\end{equation*}
where $\mathcal{P}(\mathbb{R}^{3})$ is the set of polynomials (see \cite{BCD11}).
We remark here that the Littlewood-Paley decomposition satisfies the property of almost orthogonality, that is to say, for any $f, g\in\mathcal{S}'(\mathbb{R}^{3})/\mathcal{P}(\mathbb{R}^{3})$, the following properties hold:
\begin{align}\label{eq2.2}
  \Delta_{i}\Delta_{j}f\equiv0\ \ \ \text{if}\ \ \ |i-j|\geq 2\ \ \ \text{and}\ \ \
  \Delta_{i}(S_{j-1}f\Delta_{j}g)\equiv0 \ \ \ \text{if}\ \ \ |i-j|\geq 5.
\end{align}

Using the above decomposition, the  stationary/time dependent homogeneous Besov spaces can be defined as follows:

\begin{definition}\label{de2.1}
Let $s\in \mathbb{R}$, $1\leq p,r\leq\infty$ and $f\in\mathcal{S}'(\mathbb{R}^{3})$, we set
\begin{equation*}
  \|f\|_{\dot{B}^{s}_{p,r}}:= \begin{cases} \left(\sum_{j\in\mathbb{Z}}2^{jsr}\|\Delta_{j}f\|_{L^{p}}^{r}\right)^{\frac{1}{r}}
  \ \ &\text{for}\ \ 1\leq r<\infty,\\
  \sup_{j\in\mathbb{Z}}2^{js}\|\Delta_{j}f\|_{L^{p}}\ \
  &\text{for}\ \
  r=\infty.
 \end{cases}
\end{equation*}
Then the homogeneous Besov
space $\dot{B}^{s}_{p,r}(\mathbb{R}^{3})$ is defined by
\begin{itemize}
\item For $s<\frac{3}{p}$ (or $s=\frac{3}{p}$ if $r=1$), we define
\begin{equation*}
  \dot{B}^{s}_{p,r}(\mathbb{R}^{3}):=\Big\{f\in \mathcal{S}'(\mathbb{R}^{3}):\ \
  \|f\|_{\dot{B}^{s}_{p,r}}<\infty\Big\}.
\end{equation*}
\item If $k\in\mathbb{N}$ and $\frac{3}{p}+k\leq s<\frac{3}{p}+k+1$ (or $s=\frac{3}{p}+k+1$ if $r=1$), then $\dot{B}^{s}_{p,r}(\mathbb{R}^{3})$
is defined as the subset of distributions $f\in\mathcal{S}'(\mathbb{R}^{3})$ such that $\partial^{\beta}f\in\mathcal{S}'(\mathbb{R}^{3})$
whenever $|\beta|=k$.
\end{itemize}
\end{definition}

\begin{definition}\label{de2.2} {\em (\cite{CL95})} For $0<T\leq\infty$, $s\leq \frac{3}{p}$ (resp. $s\in \mathbb{R}$),
$1\leq p, r, \rho\leq\infty$. We define the mixed time-space $\mathcal{L}^{\rho}(0,T; \dot{B}^{s}_{p,r}(\mathbb{R}^{3}))$
as the completion of $\mathcal{C}([0,T]; \mathcal{S}(\mathbb{R}^{3}))$ by the norm
$$
  \|f\|_{\mathcal{L}^{\rho}_{T}(\dot{B}^{s}_{p,r})}:=\left(\sum_{j\in\mathbb{Z}}2^{jsr}\left(\int_{0}^{T}
  \|\Delta_{j}f(\cdot,t)\|_{L^{p}}^{\rho}dt\right)^{\frac{r}{\rho}}\right)^{\frac{1}{r}}<\infty
$$
with the usual change if $\rho=\infty$
or $r=\infty$.  For simplicity, we use $\|f\|_{\mathcal{L}^{\rho}_{t}(\dot{B}^{s}_{p,r})}$ instead of $\|f\|_{\mathcal{L}^{\rho}_{\infty}(\dot{B}^{s}_{p,r})}$.
\end{definition}

The following properties of Besov spaces are well-known:

 (1) If  $s<\frac{3}{p}$ or
$s=\frac{3}{p}$ and $r=1$, then
$(\dot{B}^{s}_{p,r}(\mathbb{R}^{3}),\|\cdot\|_{\dot{B}^{s}_{p,r}})$
is a Banach space which is continuously embedded in
$\mathcal{S}'(\mathbb{R}^{3})$.

(2) In the case that $p=r=2$, we get the homogeneous Sobolev space
$\dot{H}^{s}(\mathbb{R}^{3})\cong\dot{B}^{s}_{2,2}(\mathbb{R}^{3})$,
which is endowed with the equivalent norm
$\|f\|_{\dot{H}^{s}}=\|\Lambda^{s}f\|_{L^{2}}$ with $\Lambda=\sqrt{-\Delta}$.

(3) Let $s\in \mathbb{R}$, $1\leq p,r\leq\infty$, and
$u\in\mathcal{S}'(\mathbb{R}^{3})/\mathcal{P}(\mathbb{R}^{3})$. Then $u\in
\dot{B}^{s}_{p,r}(\mathbb{R}^{3})$ if and only if there exists
$\{d_{j,r}\}_{j\in\mathbb{Z}}$ such that $d_{j,r}\ge0$,
$\|d_{j,r}\|_{l^{r}}=1$ and
$$
  \|\Delta_{j}u\|_{L^{p}}\lesssim
  d_{j,r}2^{-js}\|u\|_{\dot{B}^{s}_{p,r}} \ \ \text{for all }\
  j\in\mathbb{Z}.
$$

(4)  According to the  Minkowski inequality,
it is readily to see that
\begin{equation}\label{eq2.3}
\begin{cases}
  \|f\|_{\mathcal{L}^{\rho}_{T}(\dot{B}^{s}_{p,r})}\leq\|f\|_{L^{\rho}_{T}(\dot{B}^{s}_{p,r})} \ \ \  \text{if}\ \ \  \rho\leq r,\\
  \|f\|_{L^{\rho}_{T}(\dot{B}^{s}_{p,r})}\leq \|f\|_{\mathcal{L}^{\rho}_{T}(\dot{B}^{s}_{p,r})} \ \ \ \text{if} \ \ \  r\leq \rho.
\end{cases}
\end{equation}

Finally we recall the following Bony's paradifferential decomposition
 (see \cite{B81}). The
paraproduct between $f$ and $g$ is defined by
\begin{equation*}
  T_{f}g:=\sum_{j\in\mathbb{Z}}S_{j-1}f\Delta_{j}g.
\end{equation*}
Thus we have the formal decomposition
\begin{equation*}
  fg=T_{f}g+T_{g}f+R(f,g),
\end{equation*}
where
\begin{equation*}
  R(f,g):=\sum_{j\in\mathbb{Z}}\Delta_{j}f\widetilde{\Delta_{j}}g \ \
  \text{and}\ \
  \widetilde{\Delta_{j}}:=\Delta_{j-1}+\Delta_{j}+\Delta_{j+1}.
\end{equation*}

\section{Main results}

We are concerned with the  following system of
dissipative nonlinear equations governing
hydrodynamic transport of binary diffuse charge densities. The 3-D
Cauchy problem reads as follows:
\begin{equation}\label{eq3.1}
\begin{cases}
  \partial_{t} u+u\cdot\nabla u-\mu\Delta
  u+\nabla \Pi=\varepsilon\Delta
  \phi\nabla\phi,\ \ & x\in\mathbb{R}^{3},\ t>0,\\
  \nabla\cdot u=0,\ \ & x\in\mathbb{R}^{3},\ t>0,\\
  \partial_{t} v+u\cdot \nabla
  v=\nabla\cdot(D_{1}\nabla v-\nu_{1}v\nabla \phi),\ \ & x\in\mathbb{R}^{3},\ t>0,\\
  \partial_{t} w+u\cdot \nabla
  w=\nabla\cdot(D_{2}\nabla w+\nu_{2}w\nabla \phi),\ \ & x\in\mathbb{R}^{3},\ t>0,\\
  \varepsilon\Delta \phi=v-w,\ \ & x\in\mathbb{R}^{3},\ t>0
\end{cases}
\end{equation}
with initial condition
\begin{equation}\label{eq3.2}
  (u, v, w)|_{t=0}=(u_0, v_0, w_0), \ \  x\in\mathbb{R}^{3}.
\end{equation}
Here $u$ and $\Pi$ denote the
velocity field and the pressure of the fluid, respectively, $\phi$
is the electrostatic potential caused by the charged particles, $v$ and $w$ denote the charge
densities of a negatively and positively charged species,
respectively, hence the sign difference in front of the convective
term in either equation.  $\mu$ is the kinematic viscosity, and $\varepsilon$ is the
dielectric constant, known as the Debye length, related to vacuum
permittivity and characteristic charge density.  $D_{1}$, $D_{2}$, $\nu_{1}$, $\nu_{2}$ are
the diffusion and mobility coefficients of the charged
particles\footnote{$D_{1}=\frac{kT_{0}\nu_{1}}{e}$,
$D_{2}=\frac{kT_{0}\nu_{2}}{e}$, where $T_{0}$ is the ambient
temperature, $k$ is the Boltzmann constant, and $e$ is the charge
mobility.}. Since the concrete values of the constants $\mu$, $\varepsilon $, $D_{1}$, $D_{2}$, $\nu_{1}$ and $\nu_{2}$ play no role in our discussion,
for simplicity, we shall assume them
to be all equal to one throughout this paper.

We mention here that the right-hand side term in the momentum equations is the Lorentz force, which exhibits
$\varepsilon\Delta\phi\nabla\phi=\varepsilon\nabla\cdot\sigma$,
where the electric stress $\sigma$ is a rank one tensor plus a
pressure, for $i,j=1,2,3$,
\begin{equation}\label{eq3.3}
  [\sigma]_{ij}=\big(\nabla\phi\otimes\nabla\phi-\frac{1}{2}|\nabla\phi|^{2}I\big)_{ij}
  =\partial_{x_{i}}\phi\partial_{x_{j}}\phi-\frac{1}{2}|\nabla\phi|^{2}\delta_{ij}.
\end{equation}
Here $I$ is $3\times3$ identity matrix,  $\delta_{ij}$ is the
Kronecker symbol, and $\otimes$ denotes the tensor product. The
electric stress $\sigma$ stems from the balance of kinetic energy
with electrostatic energy via the least action principle (cf.
\cite{RLZ07}).

The system \eqref{eq3.1}--\eqref{eq3.2} was introduced by Rubinstein
\cite{R90}, which is capable of describing electro-chemical and
fluid-mechanical transport throughout the cellular environment. At
the present time, modeling of electro-diffusion in electrolytes is a
problem of major scientific interest, it finds that such model has a
wide applications in biology (ion channels), chemistry
(electro-osmosis) and pharmacology (transdermal iontophoresis), we
refer the readers to see \cite{LCJS081}--\cite{LMCJS08} for the computational simulations, and \cite{BW04}, \cite{EN00}, \cite{ES05}, \cite{SK04} for detailed  applications of the system
\eqref{eq3.1}--\eqref{eq3.2}.  The mathematical analysis of the system \eqref{eq3.1}--\eqref{eq3.2} was
initiated by  Jerome \cite{J02}, where  the author established a
local existence--uniqueness theory of the system \eqref{eq3.1}--\eqref{eq3.2}
based on the Kato's semigroup framework. For more results concerning existence of (large)
weak solutions, (small and local) mild solutions, convergence rate estimates to
stationary solutions of time-dependent solutions and other related topics we refer the reader to see \cite{DZC11}, \cite{J11}, \cite{JS09}, \cite{R09},  \cite{S09},
\cite{ZDC10}, \cite{ZZL15} and the reference therein.

The invariant space for solving the system \eqref{eq3.1}--\eqref{eq3.2} requires us to analyze the scaling invariance property of  the system \eqref{eq3.1}--\eqref{eq3.2}.  Set
\begin{equation*}
  (u_{\lambda},v_{\lambda},w_{\lambda}, \Pi_{\lambda}, \phi_{\lambda})(x,t):=(\lambda u,\lambda^2 v,\lambda^2 w, \lambda^2\Pi, \phi)(\lambda x,
  \lambda^2t).
\end{equation*}
Then  if $(u,v,w)$ solves \eqref{eq3.1} with initial data $(u_{0}, v_{0}, w_{0})$ ($\Pi, \phi$ can be determined by $(u,v,w)$), so does
$(u_{\lambda}, v_{\lambda}, w_{\lambda})$ with initial data $(u_{0\lambda}, v_{0\lambda}, w_{0\lambda})$ ($\Pi_{\lambda}, \phi_{\lambda}$ can be determined by $(u_{\lambda}, v_{\lambda}, w_{\lambda})$),  where $u_{0\lambda}(x):=\lambda u_{0}(\lambda x)$, $v_{0\lambda}(x):=\lambda^{2}v_{0}(\lambda x)$, $w_{0\lambda}(x):=\lambda^{2}w_{0}(\lambda x)$. In particular, the norm of $u_{0}\in\dot{B}^{-1+\frac{3}{p}}_{p,1}(\mathbb{R}^{3})$,  $(v_{0}, w_{0})\in\dot{B}^{-2+\frac{3}{q}}_{q,1}(\mathbb{R}^{3})$ ($1\leq p, q\leq\infty$) are scaling invariant under the above change of scale.

Motivated by the optimal time decay rates of the solutions to the heat equation in the framework of Besov spaces, we
aim at using this approach to the system \eqref{eq3.1}--\eqref{eq3.2}.  The
main results are as follows:

\begin{theorem}\label{th3.1}
Let $p,q$ be two positive numbers such that $1\leq p<\infty$,
$1\leq q<6$, and
$$
  \frac{1}{p}+\frac{1}{q}> \frac{1}{3},\ \ \ \frac{1}{q}-\frac{1}{p}> -\min\{\frac{1}{3}, \frac{1}{2p}\}.
$$
Suppose that $u_{0}\in
\dot{B}^{-1+\frac{3}{p}}_{p,1}(\mathbb{R}^{3})$ with $\nabla\cdot
u_{0}=0$, $v_{0}, w_{0}\in
\dot{B}^{-2+\frac{3}{q}}_{q,1}(\mathbb{R}^{3})$. Then  there exists a positive constant $\eta$  such that if
$$
  \|(u_{0}, v_{0}, w_{0})\|_{\dot{B}^{-1+\frac{3}{p}}_{p,1}\times(\dot{B}^{-2+\frac{3}{q}}_{q,1})^{2}}
  \leq\eta,
$$
then  the system \eqref{eq3.1}--\eqref{eq3.2} admits a unique
solution $(u,v,w)$  satisfying
\begin{equation*}
\begin{cases}
  u\in C([0,\infty),\dot{B}^{-1+\frac{3}{p}}_{p,1}(\mathbb{R}^{3}))\cap\mathcal{L}^{\infty}(0, \infty;
  \dot{B}^{-1+\frac{3}{p}}_{p,1}(\mathbb{R}^{3}))\cap  L^{1}(0, \infty;
  \dot{B}^{1+\frac{3}{p}}_{p,1}(\mathbb{R}^{3})),\\
  v,w\in C([0,\infty),\dot{B}^{-2+\frac{3}{q}}_{q,1}(\mathbb{R}^{3}))\cap\mathcal{L}^{\infty}(0, \infty;
  \dot{B}^{-2+\frac{3}{q}}_{q,1}(\mathbb{R}^{3}))\cap  L^{1}(0, \infty;
  \dot{B}^{\frac{3}{q}}_{q,1}(\mathbb{R}^{3})).
\end{cases}
\end{equation*}
If we assume further that $u_{0}\in
\dot{B}^{-s}_{r,1}(\mathbb{R}^{3})\cap\dot{B}^{N}_{r,1}(\mathbb{R}^{3})$, $v_{0}, w_{0}\in
\dot{B}^{-s-1}_{r,1}(\mathbb{R}^{3})\cap\dot{B}^{N-1}_{r,1}(\mathbb{R}^{3})$ for an integer $N$,  a real number $s>0$ and $1<r<\infty$ such that
$$
  \frac{3}{p}-s>3\max\{0, \frac{1}{p}+\frac{1}{r}-1\}\ \ \text{and}\ \ \frac{3}{q}-s>3\max\{0, \frac{1}{q}+\frac{1}{r}-1\},
$$
then for any  $\ell\in[-s, N]$,  there exists a constant $C_{0}$ such that for all $t\geq 0$,
\begin{equation}\label{eq3.4}
     \|(u(t), v(t), w(t))\|_{\dot{B}^{\ell}_{r,1}\times(\dot{B}^{\ell-1}_{r,1})^{2}}\leq C_{0}.
\end{equation}
Moreover,  we have
\begin{equation}\label{eq3.5}
     \|(u(t), v(t), w(t))\|_{\dot{B}^{\ell}_{r,1}\times(\dot{B}^{\ell-1}_{r,1})^{2}}\leq C_{0}(1+t)^{-(\frac{\ell+s}{2})}.
\end{equation}
\end{theorem}

If we relax the high regularity condition imposed on the  initial data in Theorem \ref{th3.1}, then we can obtain the following decay result.
\begin{theorem}\label{th3.2} Under the assumptions of Theorem \ref{th3.1}. Assume that  $(u,v,w)$ be a unique
global solution corresponding to the initial data  $(u_{0}, v_{0}, w_{0})$.
If we assume further that $u_{0}\in
\dot{B}^{-s}_{r,1}(\mathbb{R}^{3})$, $v_{0}, w_{0}\in
\dot{B}^{-s-1}_{r,1}(\mathbb{R}^{3})$ with  $1<r\le\min\{p,q\}$,  $s>\max\{0, 2-\frac{3}{r}\}$, and
$$
  \frac{3}{p}-s>3\max\{0, \frac{1}{p}+\frac{1}{r}-1\}\ \ \text{and}\ \ \frac{3}{q}-s>3\max\{0, \frac{1}{q}+\frac{1}{r}-1\},
$$
then for any  $\ell\in[-s-3(\frac{1}{r}-\frac{1}{p}), -1+\frac{3}{p}]$,  there exists a constant $C_{0}$ such that for all $t\geq 0$,
\begin{equation}\label{eq3.6}
     \|u(t)\|_{\dot{B}^{\ell}_{r,1}}\leq C_{0}(1+t)^{-(\frac{\ell+s}{2})-\frac{3}{2}(\frac{1}{r}-\frac{1}{p})};
\end{equation}
for any  $\ell\in[-s-1-3(\frac{1}{r}-\frac{1}{q}), -2+\frac{3}{q}]$,  there exists a constant $C_{0}$ such that for all $t\geq 0$,
\begin{equation}\label{eq3.7}
     \|(v(t), w(t))\|_{\dot{B}^{\ell-1}_{r,1}}\leq C_{0}(1+t)^{-(\frac{\ell+s}{2})-\frac{3}{2}(\frac{1}{r}-\frac{1}{q})}.
\end{equation}
\end{theorem}

We emphasize here that in \cite{ZZL15},  the authors in this paper and Zhang established global
well-posedness of the system \eqref{eq3.1}--\eqref{eq3.2} in the
critical Besov spaces
$\dot{B}^{-1+\frac{3}{p}}_{p,1}(\mathbb{R}^{3})\times(\dot{B}^{-2+\frac{3}{q}}_{q,1}(\mathbb{R}^{3}))^{2}$
with $1\leq p<\infty$ and
$1\leq q<6$, $q\leq p$ and $\frac{1}{p}+\frac{1}{q}> \frac{1}{3}$.
We relax the restrictive condition $q\leq p$  in
Theorem \ref{th3.1}. The main observation is that we can convert the estimation of  $\Delta\phi\nabla\phi$ into the estimation of
$v\nabla(-\Delta)^{-1}w+w\nabla(-\Delta)^{-1}v$ via the fifth equation of \eqref{eq3.1}, which has a nice structure as follows:  for $1\leq m\leq3$,
\begin{align*}
  \big(v\nabla(-\Delta)^{-1}w&+w\nabla(-\Delta)^{-1}v\big)_{m}=(-\Delta)\Big{\{}\big{(}(-\Delta)^{-1}v\big{)}\big{(}\partial_{m}(-\Delta)^{-1}w\big{)}\Big{\}}\\
  &+2\nabla\cdot\Big{\{}\big{(}(-\Delta)^{-1}v\big{)}\big{(}\partial_{m}\nabla(-\Delta)^{-1}w\big{)}\Big{\}}
  +\partial_{m}\Big{\{}\big{(}(-\Delta)^{-1}v\big{)}w\Big{\}}.
\end{align*}
Thanks to this observation, the condition $q\leq p$ can be removed.

Another important feature in Theorems \ref{th3.1} and \ref{th3.2} is that the negative Besov norms of the solutions of the system \eqref{eq3.1}--\eqref{eq3.2} are preserved along the time evolution and enhance the time decay rates, see Proposition \ref{pro4.6} below.

\section{Proofs of Theorems   \ref{th3.1} and \ref{th3.2}}

We aim at establishing two basic energy inequalities in the framework of Besov spaces, then prove Theorems \ref{th3.1} and \ref{th3.2}  by using the approach illustrated in Theorem \ref{th1.2}. For clarity of our statement, we leave  the proof of global well-posedness of the system \eqref{eq3.1}--\eqref{eq3.2} with small initial data  in Appendix.

\subsection{Lower-order derivative estimates}
We denote
$$
  \mathcal{E}(t):= \|u(t)\|_{\dot{B}^{-1+\frac{3}{p}}_{p,1}}+\|(v(t),w(t))\|_{\dot{B}^{-2+\frac{3}{q}}_{q,1}}
$$
and
$$
  Y(t):=\int_{0}^{t}\big( \|u(\tau)\|_{\dot{B}^{1+\frac{3}{p}}_{p,1}}+\|(v(\tau),w(\tau))\|_{\dot{B}^{\frac{3}{q}}_{q,1}}\big)d\tau.
$$

\begin{proposition}\label{pro4.1}
Let $p,q$ be two positive numbers such that $1\leq p<\infty$,
$1\leq q<6$, and
$$
  \frac{1}{p}+\frac{1}{q}> \frac{1}{3},\ \ \ \frac{1}{q}-\frac{1}{p}> -\min\{\frac{1}{3}, \frac{1}{2p}\}.
$$
 Assume that $u_{0}\in \dot{B}^{-1+\frac{3}{p}}_{p,1}(\mathbb{R}^{3})$ with $\nabla\cdot u_{0}=0$, $v_{0}, w_{0}\in \dot{B}^{-2+\frac{3}{q}}_{q,1}(\mathbb{R}^{3})$. Let $\eta$ be the number such that if
\begin{equation*}
  \|u_{0}\|_{\dot{B}^{-1+\frac{3}{p}}_{p,1}}+\|(v_{0},w_{0})\|_{\dot{B}^{-2+\frac{3}{q}}_{q,1}}\leq \eta,
\end{equation*}
then the system \eqref{eq3.1}--\eqref{eq3.2} admits a unique solution $(u,v,w)$.  In addition, there exist two constants $\kappa$ and $K$ such that the following inequality holds:
\begin{equation}\label{eq4.1}
   \frac{d}{dt}(e^{-KY(t)}\mathcal{E}(t))+\kappa e^{-KY(t)}\big(\|u(t)\|_{\dot{B}^{1+\frac{3}{p}}_{p,1}}+\|(v(t),w(t))\|_{\dot{B}^{\frac{3}{q}}_{q,1}}\big)\leq 0.
\end{equation}
\end{proposition}

In order to prove Proposition \ref{pro4.1}, we define
\begin{equation*}
  \widetilde{u}:=e^{-KY(t)}u,\ \ \widetilde{v}:= e^{-KY(t)}v,\ \ \widetilde{w}:= e^{-KY(t)}w, \ \ \widetilde{\Pi}:= e^{-KY(t)}\Pi,\ \ \widetilde{\phi}:= e^{-KY(t)}\phi,
\end{equation*}
where $K$ is a constant to be specified later. Then we see that $(\widetilde{u},  \widetilde{v}, \widetilde{w})$ satisfies the following equations:
\begin{equation}\label{eq4.2}
\begin{cases}
  \partial_{t} \widetilde{u}+u\cdot\nabla \widetilde{u}-\Delta
  \widetilde{u}+\nabla \widetilde{\Pi}=\Delta
  \widetilde{\phi}\nabla\phi-KY'(t)\widetilde{u},\\
  \nabla\cdot \widetilde{u}=0,\\
  \partial_{t} \widetilde{v}+u\cdot \nabla
  \widetilde{v}=\nabla\cdot(\nabla\widetilde{ v}-\widetilde{v}\nabla \phi)-KY'(t)\widetilde{v},\\
  \partial_{t} \widetilde{w}+u\cdot \nabla
  \widetilde{w}=\nabla\cdot(\nabla \widetilde{w}+\widetilde{w}\nabla \phi)-KY'(t)\widetilde{w},\\
  \Delta \widetilde{\phi}=\widetilde{v}-\widetilde{w}.
\end{cases}
\end{equation}

\begin{lemma}\label{le4.2}
Let $1\leq p<\infty$.  Then
\begin{align}\label{eq4.3}
   \|\Delta_{j}(u\cdot\nabla\widetilde{u})\|_{L^{p}}\lesssim  2^{(1-\frac{3}{p})j}d_{j}Y'(t)\|\widetilde{u}\|_{\dot{B}^{-1+\frac{3}{p}}_{p,1}}.
\end{align}
\end{lemma}

\begin{proof}
Thanks to Bony's paraproduct decomposition, we have
\begin{equation*}
  u\cdot\nabla\widetilde{u}=\nabla\cdot(u\otimes\widetilde{u})=\nabla\cdot\big(2T_{\widetilde{u}}u+R(\tilde{u},u)\big).
\end{equation*}
Moreover, applying Lemma \ref{le5.1} yields that
\begin{align*}
  \|\Delta_{j}\nabla\cdot(T_{\widetilde{u}}u)\|_{L^{p}}&\lesssim  2^{j}\sum_{|j-j'|\leq 4}\|S_{j'-1}\widetilde{u}\|_{L^{\infty}}\|\Delta_{j'}u\|_{L^{p}}\nonumber\\
   &\lesssim  2^{j}\sum_{|j-j'|\leq 4}\sum_{k\leq j'-2}2^{k}2^{(-1+\frac{3}{p})k}\|\Delta_{k}\widetilde{u}\|_{L^{p}}\|\Delta_{j'}u\|_{L^{p}}\nonumber\\
   &\lesssim  2^{j}\sum_{|j-j'|\leq 4}2^{j'}\|\Delta_{j'}u\|_{L^{p}}\|\widetilde{u}\|_{\dot{B}^{-1+\frac{3}{p}}_{p,1}}\nonumber\\
   &\lesssim  2^{(1-\frac{3}{p})j}d_{j}\|u\|_{\dot{B}^{1+\frac{3}{p}}_{p,1}}\|\widetilde{u}\|_{\dot{B}^{-1+\frac{3}{p}}_{p,1}}\nonumber\\
   &\lesssim  2^{(1-\frac{3}{p})j}d_{j}Y'(t)\|\widetilde{u}\|_{\dot{B}^{-1+\frac{3}{p}}_{p,1}}.
\end{align*}
To estimate the remaining term $R(\widetilde{u},u)$, in the case $1\leq p<2$, there exists $2<p'\leq\infty$
such that $\frac{1}{p}+\frac{1}{p'}=1$, thus we can deduce from Lemma \ref{le5.1}  that
\begin{align*}
  \|\Delta_{j}\nabla\cdot R(\widetilde{u}, u)\|_{L^{p}}&\lesssim
   2^{(4-\frac{3}{p})j}\sum_{j'\geq j-N_{0}}\|\Delta_{j'}\widetilde{u}\|_{L^{p'}}\|\widetilde{\Delta}_{j'}u\|_{L^{p}}\nonumber\\
   &\lesssim  2^{(4-\frac{3}{p})j}\sum_{j'\geq j-N_{0}}2^{(-3+\frac{6}{p})j'}\|\Delta_{j'}\widetilde{u}\|_{L^{p}}\|\widetilde{\Delta}_{j'}u\|_{L^{p}}\nonumber\\
   &\lesssim  2^{(1-\frac{3}{p})j}\sum_{j'\geq j-N_{0}}2^{-3(j'-j)}2^{(-1+\frac{3}{p})j'}\|\Delta_{j'}\widetilde{u}\|_{L^{p}}2^{(1+\frac{3}{p})j'}\|\widetilde{\Delta}_{j'}u\|_{L^{p}}\nonumber\\
   &\lesssim  2^{(1-\frac{3}{p})j}d_{j}\|u\|_{\dot{B}^{1+\frac{3}{p}}_{p,1}}\|\widetilde{u}\|_{\dot{B}^{-1+\frac{3}{p}}_{p,1}}\nonumber\\
   &\lesssim  2^{(1-\frac{3}{p})j}d_{j}Y'(t)\|\widetilde{u}\|_{\dot{B}^{-1+\frac{3}{p}}_{p,1}}.
\end{align*}
If $2\leq p<\infty$, we estimate
\begin{align*}
  \|\Delta_{j}\nabla\cdot R(\widetilde{u}, u)\|_{L^{p}}&\lesssim
   2^{(1+\frac{3}{p})j}\sum_{j'\geq j-N_{0}}\|\Delta_{j'}\widetilde{u}\|_{L^{p}}\|\widetilde{\Delta}_{j'}u\|_{L^{p}}\nonumber\\
   &\lesssim  2^{(1-\frac{3}{p})j}\sum_{j'\geq j-N_{0}}2^{-\frac{6}{p}(j'-j)}2^{(-1+\frac{3}{p})j'}\|\Delta_{j'}\widetilde{u}\|_{L^{p}}2^{(1+\frac{3}{p})j'}\|\widetilde{\Delta}_{j'}u\|_{L^{p}}\nonumber\\
   &\lesssim  2^{(1-\frac{3}{p})j}d_{j}Y'(t)\|\widetilde{u}\|_{\dot{B}^{-1+\frac{3}{p}}_{p,1}}.
\end{align*}
This completes the proof of Lemma \ref{le4.2}.
\end{proof}

\begin{lemma}\label{le4.3}
Let $1\leq p, q<\infty$ and $\frac{1}{q}-\frac{1}{p}\geq-\min\{\frac{1}{3}, \frac{1}{2p}\}$.  Then
\begin{align}\label{eq4.4}
   \|\Delta_{j}(\widetilde{v}\nabla(-\Delta)^{-1}w+\widetilde{w}\nabla(-\Delta)^{-1}v)\|_{L^{p}}\lesssim  2^{(1-\frac{3}{p})j}d_{j}Y'(t)\|(\widetilde{v},\widetilde{w})\|_{\dot{B}^{-2+\frac{3}{q}}_{q,1}}.
\end{align}
\end{lemma}
\begin{proof}
The case $1\leq q\leq p$ is simple. Indeed, based on the observation
\begin{equation*}
    v\nabla(-\Delta)^{-1}w+w\nabla(-\Delta)^{-1}v=\nabla\cdot\big(  \nabla(-\Delta)^{-1}v\nabla(-\Delta)^{-1}w\big)
\end{equation*}
and the imbedding relation $\dot{B}^{-1+\frac{3}{q}}_{q,1}(\mathbb{R}^{3})\hookrightarrow\dot{B}^{-1+\frac{3}{p}}_{p,1}(\mathbb{R}^{3})$, we see that
$ \nabla(-\Delta)^{-1}v$ play the same role as $u$ in Lemma \ref{le4.2}. Therefore, we get the desired inequality \eqref{eq4.4}. On the other hand, if  $1\leq p<q$, we resort to Bony's paraproduct decomposition to get
\begin{equation}\label{eq4.5}
    \widetilde{v}\nabla(-\Delta)^{-1}w+\widetilde{w}\nabla(-\Delta)^{-1}v:=J_{1}+J_{2}+J_{3},
\end{equation}
where
\begin{align*}
    J_{1}:&=\sum_{j'\in\mathbb{Z}}S_{j'-1}\widetilde{v}\nabla(-\Delta)^{-1}\Delta_{j'}w
    +S_{j'-1}\widetilde{w}\nabla(-\Delta)^{-1}\Delta_{j'}v,\\
    J_{2}:&=\sum_{j'\in\mathbb{Z}}\Delta_{j'}v\nabla(-\Delta)^{-1}S_{j'-1}\widetilde{w}
    +\Delta_{j'}w\nabla(-\Delta)^{-1}S_{j'-1}\widetilde{v},\\
    J_{3}:&=\sum_{j'\in\mathbb{Z}}\Delta_{j'}\widetilde{v}\nabla(-\Delta)^{-1}\widetilde{\Delta}_{j'}w
    +\Delta_{j'}\widetilde{w}\nabla(-\Delta)^{-1}\widetilde{\Delta}_{j'}v.
\end{align*}
For $J_{1}$, it suffices to deal with the first term  $\sum_{j'\in\mathbb{Z}}S_{j'-1}\widetilde{v}\nabla(-\Delta)^{-1}\Delta_{j'}w$ because of the second one can be done analogously. Using the conditions $1\leq p<q<\infty$ and $\frac{1}{q}-\frac{1}{p}\geq-\min\{\frac{1}{3}, \frac{1}{2p}\}$,  we derive from  Lemma \ref{le5.1} that
\begin{align*}
   \|\Delta_{j}\sum_{j'\in\mathbb{Z}}S_{j'-1}\widetilde{v}\nabla(-\Delta)^{-1}\Delta_{j'}w\|_{L^{p}} &\lesssim \sum_{|j-j'|\leq 4}\|S_{j'-1}\widetilde{v}\|_{L^{\frac{pq}{q-p}}}\|\nabla(-\Delta)^{-1}\Delta_{j'}w\|_{L^{q}}\nonumber\\
    &\lesssim \sum_{|j-j'|\leq 4}\sum_{k\leq j'-2}2^{3(\frac{1}{q}-\frac{q-p}{pq})k}\|\Delta_{k}\widetilde{v}\|_{L^{q}}2^{-j'}\|\Delta_{j'}w\|_{L^{q}}\nonumber\\
  & \lesssim \sum_{|j-j'|\leq 4}\sum_{k\leq j'-2}2^{(2+\frac{3}{q}-\frac{3}{p})k}2^{(-2+\frac{3}{q})k}\|\Delta_{k}\widetilde{v}\|_{L^{q}}2^{-j'}\|\Delta_{j'}w\|_{L^{q}}\nonumber\\
  &\lesssim \sum_{|j-j'|\leq 4}2^{(1+\frac{3}{q}-\frac{3}{p})j'}\|\Delta_{j'}w\|_{L^{q}}\|\widetilde{v}\|_{\dot{B}^{-2+\frac{3}{q}}_{q,1}}\nonumber\\
  &\lesssim 2^{(1-\frac{3}{p})j}d_{j}\|w\|_{\dot{B}^{\frac{3}{q}}_{q,1}}\|\widetilde{v}\|_{\dot{B}^{-2+\frac{3}{q}}_{q,1}}\nonumber\\
  &\lesssim 2^{(1-\frac{3}{p})j}d_{j}Y'(t)\|\widetilde{v}\|_{\dot{B}^{-2+\frac{3}{q}}_{q,1}},
\end{align*}
which directly leads to
\begin{align}\label{eq4.6}
  \|\Delta_{j}J_{1}\|_{L^{p}}\lesssim
       2^{(1-\frac{3}{p})j}d_{j}Y'(t)\|(\widetilde{v}, \widetilde{w})\|_{\dot{B}^{-2+\frac{3}{q}}_{q,1}}.
\end{align}
Similarly, for the first term of $J_{2}$,  we get
\begin{align*}
    \|\Delta_{j}\sum_{j'\in\mathbb{Z}}\Delta_{j'}v\nabla(-\Delta)^{-1}S_{j'-1}\widetilde{w}\|_{L^{p}} &\lesssim \sum_{|j-j'|\leq 4}\|\Delta_{j'}v\|_{L^{q}}\|\nabla(-\Delta)^{-1}S_{j'-1}\widetilde{w}\|_{L^{\frac{pq}{q-p}}}\nonumber\\
     & \lesssim \sum_{|j-j'|\leq 4}\|\Delta_{j'}v\|_{L^{q}}\sum_{k\leq j'-2}2^{[-1+3(\frac{1}{q}-\frac{q-p}{pq})]k}\|\Delta_{k}\widetilde{w}\|_{L^{q}}\nonumber\\
  &\lesssim \sum_{|j-j'|\leq 4}\|\Delta_{j'}v\|_{L^{q}}\sum_{k\leq j'-2}2^{(1+\frac{3}{q}-\frac{3}{p})k}2^{(-2+\frac{3}{q})k}\|\Delta_{k}\widetilde{w}\|_{L^{q}}\nonumber\\
  &\lesssim \sum_{|j-j'|\leq 4}2^{(1+\frac{3}{q}-\frac{3}{p})j'}\|\Delta_{j'}v\|_{L^{q}}\|\widetilde{w}\|_{\dot{B}^{-2+\frac{3}{q}}_{q,1}}\nonumber\\
  &\lesssim 2^{(1-\frac{3}{p})j}d_{j}\|v\|_{\dot{B}^{\frac{3}{q}}_{q,1}}\|\widetilde{w}\|_{\dot{B}^{-2+\frac{3}{q}}_{q,1}}\nonumber\\
  &\lesssim 2^{(1-\frac{3}{p})j}d_{j}Y'(t)\|\widetilde{w}\|_{\dot{B}^{-2+\frac{3}{q}}_{q,1}},
\end{align*}
which yields that
\begin{align}\label{eq4.7}
     \|\Delta_{j}J_{2}\|_{L^{p}}\lesssim 2^{(1-\frac{3}{p})j}d_{j}Y'(t)\|(\widetilde{v}, \widetilde{w})\|_{\dot{B}^{-2+\frac{3}{q}}_{q,1}}.
\end{align}
Finally we tackle with the most difficult term $J_{3}$,  the interesting observation is that we can split $J_{3}$ into the following three terms for $m=1,2,3$:
\begin{equation}\label{eq4.8}
  J_{3}:=K_{1}+K_{2}+K_{3},
\end{equation}
where
\begin{align*}
    K_{1}:&=\sum_{j'\in\mathbb{Z}}(-\Delta)\Big{\{}\big{(}(-\Delta)^{-1}\Delta_{j'}\widetilde{v}\big{)}\big{(}\partial_{m}(-\Delta)^{-1}\widetilde{\Delta}_{j'}w\big{)}\Big{\}},\\
    K_{2}:&=\sum_{j'\in\mathbb{Z}}2\nabla\cdot\Big{\{}\big{(}(-\Delta)^{-1}\Delta_{j'}\widetilde{v}\big{)}\big{(}\partial_{m}\nabla(-\Delta)^{-1}\widetilde{\Delta}_{j'}w\big{)}\Big{\}},\\
    K_{3}:&=\sum_{j'\in\mathbb{Z}}\partial_{m}\Big{\{}\big{(}(-\Delta)^{-1}\Delta_{j'}\widetilde{v}\big{)}\widetilde{\Delta}_{j'}w\Big{\}}.
\end{align*}
Since $K_{2}$ can be treated similarly to $K_{3}$, we treat $K_{1}$ and $K_{3}$
only. It follows from Lemma \ref{le5.1} that
\begin{align*}
    \|\Delta_{j}K_{1}\|_{L^{p}}&\lesssim 2^{2j}\sum_{j'\geq j-N_{0}}
    \|(-\Delta)^{-1}\Delta_{j'}\widetilde{v}\|_{L^{\frac{pq}{q-p}}}\|\partial_{m}(-\Delta)^{-1}\widetilde{\Delta}_{j'}w\|_{L^{q}}\nonumber\\
    &\lesssim 2^{2j}\sum_{j'\geq j-N_{0}}2^{(-2+\frac{6}{q}-\frac{3}{p})j'}
    \|\Delta_{j'}\widetilde{v}\|_{L^{q}}2^{-j'}\|\widetilde{\Delta}_{j'}w\|_{L^{q}}\nonumber\\
    &\lesssim 2^{2j}\sum_{j'\geq j-N_{0}}2^{-(1+\frac{3}{p})j'}2^{(-2+\frac{3}{q})j'}
    \|\Delta_{j'}\widetilde{v}\|_{L^{q}}2^{\frac{3j'}{q}}\|\widetilde{\Delta}_{j'}w\|_{L^{q}}\nonumber\\
    &\lesssim 2^{(1-\frac{3}{p})j}d_{j}\|w\|_{\dot{B}^{\frac{3}{q}}_{q,1}}\|\widetilde{v}\|_{\dot{B}^{-2+\frac{3}{q}}_{q,1}}\nonumber\\
    &\lesssim 2^{(1-\frac{3}{p})j}d_{j}Y'(t)\|\widetilde{v}\|_{\dot{B}^{-2+\frac{3}{q}}_{q,1}},
\end{align*}
\begin{align*}
    \|\Delta_{j}K_{3}\|_{L^{p}}&\lesssim 2^{j}\sum_{j'\geq j-N_{0}}
    \|(-\Delta)^{-1}\Delta_{j'}\widetilde{v}\|_{L^{\frac{pq}{q-p}}}\|\widetilde{\Delta}_{j'}w\|_{L^{q}}\nonumber\\
   &\lesssim 2^{j}\sum_{j'\geq j-N_{0}}2^{-\frac{3}{p}j'}2^{(-2+\frac{3}{q})j'}
    \|\Delta_{j'}\widetilde{v}\|_{L^{q}}2^{\frac{3}{q}j'}\|\widetilde{\Delta}_{j'}w\|_{L^{q}}\nonumber\\
    &\lesssim 2^{(1-\frac{3}{p})j}d_{j}\|w\|_{\dot{B}^{\frac{3}{q}}_{q,1}}\|\widetilde{v}\|_{\dot{B}^{-2+\frac{3}{q}}_{q,1}}\nonumber\\
    &\lesssim 2^{(1-\frac{3}{p})j}d_{j}Y'(t)\|\widetilde{v}\|_{\dot{B}^{-2+\frac{3}{q}}_{q,1}}.
\end{align*}
As a consequence,  we deduce from \eqref{eq4.8} that
\begin{align}\label{eq4.9}
      \|\Delta_{j} J_{3}\|_{L^{p}}\lesssim
      2^{(1-\frac{3}{p})j}d_{j}Y'(t)\|\widetilde{v}\|_{\dot{B}^{-2+\frac{3}{q}}_{q,1}}.
\end{align}
Hence, plugging \eqref{eq4.6}, \eqref{eq4.7} and \eqref{eq4.9} into \eqref{eq4.5}, we obtain \eqref{eq4.4}. The proof of Lemma \ref{le4.3}
is complete.
\end{proof}
\begin{lemma}\label{le4.4}
Let $1\leq p,q<\infty$ and
$\frac{1}{p}+\frac{1}{q}>\frac{1}{3}$. Then we have
\begin{equation}\label{eq4.10}
   \|\Delta_{j}(u\cdot\nabla
   \widetilde{v})\|_{L^{q}}\lesssim 2^{(2-\frac{3}{q})j}d_{j}Y'(t)\big(\|\widetilde{u}\|_{\dot{B}^{-1+\frac{3}{p}}_{p,1}}+
   \|\widetilde{v}\|_{\dot{B}^{-2+\frac{3}{q}}_{q,1}}\big).
\end{equation}
\end{lemma}
\begin{proof}
Thanks to Bony's paraproduct decomposition, we have
\begin{equation*}
   u\cdot\nabla \widetilde{v}=T_{\widetilde{u}}\nabla v+T_{\nabla \widetilde{v}}u+R(u,\nabla \widetilde{v}).
\end{equation*}
Applying Lemma \ref{le5.1} gives us to
\begin{align*}
  \|\Delta_{j}(T_{\widetilde{u}}\nabla v)\|_{L^{q}}
  &\lesssim \sum_{|j'-j|\le4}2^{j'}\|S_{j'-1}\widetilde{u}\|_{L^{\infty}}\|\Delta_{j'}v\|_{L^{q}}\nonumber\\
  &\lesssim \sum_{|j'-j|\le4}2^{j'}\sum_{k\le j'-2}2^{\frac{3k}{p}}\|\Delta_{k}\widetilde{u}\|_{L^{p}}\|\Delta_{j'}v\|_{L^{q}}
  \nonumber\\
   &\lesssim \sum_{|j'-j|\le4}2^{2j'}\|\Delta_{j'}v\|_{L^{q}}\|\widetilde{u}\|_{\dot{B}^{-1+\frac{3}{p}}_{p,1}}\nonumber\\
   &\lesssim 2^{(2-\frac{3}{q})j}d_{j}\|v\|_{\dot{B}^{\frac{3}{q}}_{q,1}}
  \|\widetilde{u}\|_{\dot{B}^{-1+\frac{3}{p}}_{p,1}}
  \nonumber\\
  &\lesssim 2^{(2-\frac{3}{q})j}d_{j}Y'(t)\|\widetilde{u}\|_{\dot{B}^{-1+\frac{3}{p}}_{p,1}}.
\end{align*}
If $1\leq q\leq p$,  then there exists $1<\lambda\leq \infty$ such that $\frac{1}{q}=\frac{1}{p}+\frac{1}{\lambda}$, we calculate as
\begin{align*}
  \|\Delta_{j}(T_{\nabla \widetilde{v}}u)\|_{L^{q}}&\lesssim
  \sum_{|j'-j|\le4}\|\Delta_{j}S_{j'-1}\nabla \widetilde{v}\|_{L^{\lambda}}\|\Delta_{j'}u\|_{L^{p}}\nonumber\\
  &\lesssim \sum_{|j'-j|\le4}\sum_{k\le j'-2}2^{(1+\frac{3}{p})k}\|\Delta_{k} \widetilde{v}\|_{L^{q}}
  \|\Delta_{j'}u\|_{L^{p}}
  \nonumber\\
   &\lesssim \sum_{|j'-j|\le4}\sum_{k\le j'-2}2^{(3+\frac{3}{p}-\frac{3}{q})k}2^{(-2+\frac{3}{q})k}\|\Delta_{k} \widetilde{v}\|_{L^{q}}
  \|\Delta_{j'}u\|_{L^{p}}
   \nonumber\\
  &\lesssim 2^{(2-\frac{3}{q})j}d_{j}
  \|u\|_{\dot{B}^{1+\frac{3}{p}}_{p,1}}\|\widetilde{v}\|_{\dot{B}^{-2+\frac{3}{q}}_{q,1}}
  \nonumber\\
  &\lesssim 2^{(2-\frac{3}{q})j}d_{j}Y'(t)\|\widetilde{v}\|_{\dot{B}^{-2+\frac{3}{q}}_{q,1}}.
\end{align*}
If  $1\leq p<q$,  we calculate as
\begin{align*}
  \|\Delta_{j}(T_{\nabla \widetilde{v}}u)\|_{L^{q}}&\lesssim 2^{3(\frac{1}{p}-\frac{1}{q})j}
  \sum_{|j'-j|\le4}\|S_{j'-1}\nabla \widetilde{v}\Delta_{j'}u\|_{L^{p}}\nonumber\\
  &\lesssim 2^{3(\frac{1}{p}-\frac{1}{q})j}\sum_{|j'-j|\le4}\sum_{k\le j'-2}2^{(1+\frac{3}{q})k}\|\Delta_{k} \widetilde{v}\|_{L^{q}}
  \|\Delta_{j'}u\|_{L^{p}}
  \nonumber\\
   &\lesssim 2^{3(\frac{1}{p}-\frac{1}{q})j}\sum_{|j'-j|\le4}\sum_{k\le j'-2}2^{3k}2^{(-2+\frac{3}{q})k}\|\Delta_{k} \widetilde{v}\|_{L^{q}}
  \|\Delta_{j'}u\|_{L^{p}}
  \nonumber\\
  &\lesssim 2^{(2-\frac{3}{q})j} d_{j}
  \|u\|_{\dot{B}^{1+\frac{3}{p}}_{p,1}}\|\widetilde{v}\|_{\dot{B}^{-2+\frac{3}{q}}_{q,1}}
  \nonumber\\
  &\lesssim 2^{(2-\frac{3}{q})j}d_{j}Y'(t)\|\widetilde{v}\|_{\dot{B}^{-2+\frac{3}{q}}_{q,1}}.
\end{align*}
To estimate the remaining term $R(u,\nabla \widetilde{v})$, in the case that $\frac{1}{p}+\frac{1}{q}\leq
1$, the condition $\frac{1}{p}+\frac{1}{q}> \frac{1}{3}$ implies that
\begin{align*}
  \|\Delta_{j}R(u,\nabla \widetilde{v})\|_{L^{q}}
  &\lesssim 2^{(1+\frac{3}{p})j}\sum_{j'\ge j-N_{0}}\|\Delta_{j'}u\|_{L^{p}}\|\widetilde{\Delta}_{j'}\widetilde{v}\|_{L^{q}}\nonumber\\
  &\lesssim 2^{(1+\frac{3}{p})j}\sum_{j'\ge j-N_{0}}2^{(1-\frac{3}{p}-\frac{3}{q})j'}2^{(1+\frac{3}{p})j'}\|\Delta_{j'}u\|_{L^{p}}
  2^{(-2+\frac{3}{q})j'}\|\widetilde{\Delta}_{j'}\widetilde{v}\|_{L^{q}}\nonumber\\
  &\lesssim 2^{(2-\frac{3}{q})j}d_{j}\|u\|_{\dot{B}^{1+\frac{3}{p}}_{p,1}}
  \|\widetilde{v}\|_{\dot{B}^{-2+\frac{3}{q}}_{q,1}}\nonumber\\
   &\lesssim 2^{(2-\frac{3}{q})j}d_{j}Y'(t)\|\widetilde{v}\|_{\dot{B}^{-2+\frac{3}{q}}_{q,1}}.
\end{align*}
In the case that $\frac{1}{p}+\frac{1}{q}>1$, we find
$1<q'\le\infty$ such that $\frac{1}{q}+\frac{1}{q'}=1$,
\begin{align*}
  \|\Delta_{j}R(u,\nabla \widetilde{v})\|_{L^{q}}&\lesssim
  2^{j+3(1-\frac{1}{q})j}\sum_{j'\ge j-N_{0}}\|\Delta_{j'}u\widetilde{\Delta}_{j'}\widetilde{v}\|_{L^{1}}\nonumber\\
  &\lesssim 2^{(4-\frac{3}{q})j}\sum_{j'\ge j-N_{0}}\|\Delta_{j'}u\|_{L^{q'}}\|\widetilde{\Delta}_{j'}\widetilde{v}\|_{L^{q}}\nonumber\\
  &\lesssim 2^{(4-\frac{3}{q})j}\sum_{j'\ge j-N_{0}}2^{-2j'}2^{(1+\frac{3}{p})j'}\|\Delta_{j'}u\|_{L^{p}}
  2^{(-2+\frac{3}{q})j'}\|\widetilde{\Delta}_{j'}\widetilde{v}\|_{L^{q}}\nonumber\\
  &\lesssim 2^{(2-\frac{3}{q})j}d_{j}\|u\|_{\dot{B}^{1+\frac{3}{p}}_{p,1}}
  \|\widetilde{v}\|_{\dot{B}^{-2+\frac{3}{q}}_{q,1}}\nonumber\\
   &\lesssim 2^{(2-\frac{3}{q})j}d_{j}Y'(t)\|\widetilde{v}\|_{\dot{B}^{-2+\frac{3}{q}}_{q,1}}.
\end{align*}
We finish the proof of Lemma
\ref{le4.4}.
\end{proof}

\begin{lemma}\label{le4.5}
Let $1\leq q<6$.  Then we have
\begin{equation}\label{eq4.11}
   \|\Delta_{j}(\widetilde{v}\nabla(-\Delta)^{-1}w)\|_{L^{q}}
   \lesssim 2^{(1-\frac{3}{q})j}d_{j}Y'(t)\|(\widetilde{v}, \widetilde{w})\|_{\dot{B}^{-2+\frac{3}{q}}_{q,1}}.
\end{equation}
\end{lemma}
\begin{proof}
Thanks to Bony's paraproduct decomposition, we obtain
\begin{equation*}
   \widetilde{v}\nabla(-\Delta)^{-1}w=T_{\widetilde{v}}\nabla(-\Delta)^{-1}w+T_{\nabla(-\Delta)^{-1}\widetilde{w}}v+R(\widetilde{v},\nabla(-\Delta)^{-1}w).
\end{equation*}
Applying Lemmas \ref{le5.1} and \ref{le5.2}  yields that
\begin{align*}
  \|\Delta_{j}(T_{\widetilde{v}}\nabla(-\Delta)^{-1}w)\|_{L^{q}}
  &\lesssim \sum_{|j'-j|\le4}\|S_{j'-1}\widetilde{v}\|_{L^{\infty}}\|\Delta_{j'}\nabla(-\Delta)^{-1}w\|_{L^{q}}\nonumber\\
  &\lesssim \sum_{|j'-j|\le4}2^{-j'}\sum_{k\leq j'-2}2^{\frac{3k}{q}}\|\Delta_{k}\widetilde{v}\|_{L^{q}}\|\Delta_{j'}w\|_{L^{q}}\nonumber\\
  &\lesssim \sum_{|j'-j|\le4}2^{(-1-\frac{3}{q})j'}\sum_{k\leq j'-2}2^{2k}2^{(-2+\frac{3}{q})k}
  \|\Delta_{k}\widetilde{v}\|_{L^{q}}2^{\frac{3j'}{q}}
  \|\Delta_{j'}w\|_{L^{q}}\nonumber\\
  &\lesssim 2^{(1-\frac{3}{q})j}d_{j}\|w\|_{\dot{B}^{\frac{3}{q}}_{q,1}}\|\widetilde{v}\|_{\dot{B}^{-2+\frac{3}{q}}_{q,1}}
  \nonumber\\
  &\lesssim 2^{(1-\frac{3}{q})j}d_{j}Y'(t)\|\widetilde{v}\|_{\dot{B}^{-2+\frac{3}{q}}_{q,1}},
\end{align*}
\begin{align*}
  \|\Delta_{j}(T_{\nabla(-\Delta)^{-1}\widetilde{w}}v)\|_{L^{q}}
  &\lesssim \sum_{|j'-j|\le4}\|S_{j'-1}\nabla(-\Delta)^{-1}\widetilde{w}\|_{L^{\infty}}\|\Delta_{j'}v\|_{L^{q}}\nonumber\\
  &\lesssim \sum_{|j'-j|\le4}\sum_{k\leq j'-2}2^{(-1+\frac{3}{q})k}\|\Delta_{k}\widetilde{w}\|_{L^{q}}\|\Delta_{j'}v\|_{L^{q}}\nonumber\\
  &\lesssim \sum_{|j'-j|\le4}2^{-\frac{3j'}{q}}\sum_{k\leq j'-2}2^{k}2^{(-2+\frac{3}{q})k}
  \|\Delta_{k}\widetilde{w}\|_{L^{q}}2^{\frac{3j'}{q}}
  \|\Delta_{j'}v\|_{L^{q}}\nonumber\\
  &\lesssim 2^{(1-\frac{3}{q})j}d_{j}\|v\|_{\dot{B}^{\frac{3}{q}}_{q,1}}\|\widetilde{w}\|_{\dot{B}^{-2+\frac{3}{q}}_{q,1}}
  \nonumber\\
  &\lesssim 2^{(1-\frac{3}{q})j}d_{j}Y'(t)\| \widetilde{w}\|_{\dot{B}^{-2+\frac{3}{q}}_{q,1}}.
\end{align*}
Finally, in the case $1\leq q<2$, there exists $2<q'\le\infty$ such that
$\frac{1}{q}+\frac{1}{q'}=1$, thus using Lemma \ref{le5.1}
yields that
\begin{align*}
  \|\Delta_{j}R(\widetilde{v},\nabla(-\Delta)^{-1}w)\|_{L^{q}}&\lesssim
  2^{3(1-\frac{1}{q})j}\sum_{j'\ge j-N_{0}}\|\Delta_{j'}\widetilde{v}\widetilde{\Delta}_{j'}\nabla(-\Delta)^{-1}w\|_{L^{1}}\nonumber\\
  &\lesssim 2^{3(1-\frac{1}{q})j}\sum_{j'\ge j-N_{0}}\|\Delta_{j'}\widetilde{v}\|_{L^{q'}}
  \|\widetilde{\Delta}_{j'}\nabla(-\Delta)^{-1}w\|_{L^{q}}\nonumber\\
  &\lesssim 2^{3(1-\frac{1}{q})j}\sum_{j'\ge j-N_{0}}2^{-2j'}2^{(-2+\frac{3}{q})j'}\|\Delta_{j'}\widetilde{v}\|_{L^{q}}
  2^{\frac{3j'}{q}}\|\widetilde{\Delta}_{j'}w\|_{L^{q}}\nonumber\\
  &\lesssim 2^{(1-\frac{3}{q})j}d_{j}\|w\|_{\dot{B}^{\frac{3}{q}}_{q,1}}\|\widetilde{v}\|_{\dot{B}^{-2+\frac{3}{q}}_{q,1}}
  \nonumber\\
  &\lesssim 2^{(1-\frac{3}{q})j}d_{j}Y'(t)\|\widetilde{v}\|_{\dot{B}^{-2+\frac{3}{q}}_{q,1}}.
\end{align*}
In the case $2\leq
q<6$, we get by using Lemma \ref{le5.1} again that
\begin{align*}
  \|\Delta_{j}R(\widetilde{v},\nabla(-\Delta)^{-1}w)\|_{L^{q}}
  &\lesssim 2^{\frac{3j}{q}}\sum_{j'\ge j-N_{0}}\|\Delta_{j'}\widetilde{v}\|_{L^{q}}
  \|\widetilde{\Delta}_{j'}\nabla(-\Delta)^{-1}w\|_{L^{q}}\nonumber\\
  &\lesssim 2^{\frac{3j}{q}}\sum_{j'\ge j-N_{0}}2^{(1-\frac{6}{q})j'}2^{(-2+\frac{3}{q})j'}\|\Delta_{j'}\widetilde{v}\|_{L^{q}}
  2^{\frac{3j'}{q}}\|\widetilde{\Delta}_{j'}w\|_{L^{q}}\nonumber\\
  &\lesssim 2^{(1-\frac{3}{q})j}d_{j}\|w\|_{\dot{B}^{\frac{3}{q}}_{q,1}}\|\widetilde{v}\|_{\dot{B}^{-2+\frac{3}{q}}_{q,1}}
 \nonumber\\
  &\lesssim 2^{(1-\frac{3}{q})j}d_{j}Y'(t) \|\widetilde{v}\|_{\dot{B}^{-2+\frac{3}{q}}_{q,1}}.
\end{align*}

We conclude that the proof of
Lemma \ref{le4.5} is complete.
\end{proof}

\medskip

 \textbf{The estimates of $u$}
 Applying the dyadic operator $\Delta_{j}$ to the first equation of \eqref{eq4.2}, then multiplying $|\Delta_{j}\widetilde{u}|^{p-2}\Delta_{j}\widetilde{u}$ and integrating over $\mathbb{R}^{3}$ (when
$p\in(1,2)$, we need to make some modification as that in
\cite{D01}), we see that
\begin{align}\label{eq4.12}
  \frac{1}{p}\frac{d}{dt}\|\Delta_{j}\widetilde{u}\|_{L^{p}}^{p}
  &-\int_{\mathbb{R}^{3}}\Delta\Delta_{j}\widetilde{u}|\Delta_{j}\widetilde{u}|^{p-2}\Delta_{j}\widetilde{u}dx=
  -\big(\Delta_{j}(u\cdot\nabla\widetilde{u})\big{|}|\Delta_{j}\widetilde{u}|^{p-2}\Delta_{j}\widetilde{u}\big)\nonumber\\
  &+\big(\Delta_{j}(\Delta\widetilde{\phi}\nabla\phi)\big{|}|\Delta_{j}\widetilde{u}|^{p-2}\Delta_{j}\widetilde{u}\big)
  -KY'(t)\|\Delta_{j}\widetilde{u}\|_{L^{p}}^{p}\nonumber\\
  &\leq \|\Delta_{j}(u\cdot\nabla\widetilde{u})\|_{L^{p}}\|\Delta_{j}\widetilde{u}\|_{L^{p}}^{p-1}
  +\|\Delta_{j}(\Delta\widetilde{\phi}\nabla\phi)\|_{L^{p}}\|\Delta_{j}\widetilde{u}\|_{L^{p}}^{p-1}-KY'(t)\|\Delta_{j}\widetilde{u}\|_{L^{p}}^{p},
\end{align}
where we have used the fact
\begin{equation*}
  \int_{\mathbb{R}^{3}}\nabla\Delta_{j}\widetilde{\Pi}|\Delta_{j}\widetilde{u}|^{p-2}\Delta_{j}\widetilde{u}dx=0,
\end{equation*}
which follows from the incompressibility condition $\nabla\cdot\widetilde{u}=0$.
Thanks to \cite{D01,P00}, there exists a positive constant $\kappa$
so that
\begin{equation*}
  -\int_{\mathbb{R}^{3}}\Delta\Delta_{j}\widetilde{u}\cdot|\Delta_{j}\widetilde{u}|^{p-2}\Delta_{j}\widetilde{u}dx\geq
  \kappa2^{2j}\|\Delta_{j}\widetilde{u}\|_{L^{p}}^{p}.
\end{equation*}
Therefore, we infer from \eqref{eq4.12} that
\begin{align*}
  \frac{d}{dt}\|\Delta_{j}\widetilde{u}\|_{L^{p}}+\kappa2^{2j}\|\Delta_{j}\widetilde{u}\|_{L^{p}}
   &\lesssim \|\Delta_{j}(u\cdot\nabla\widetilde{u})\|_{L^{p}}
  +\|\Delta_{j}(\Delta\widetilde{\phi}\nabla\phi)\|_{L^{p}}-KY'(t)\|\Delta_{j}\widetilde{u}\|_{L^{p}}.
\end{align*}
Note that by the Poisson equation, i.e., the fifth equation  of the system \eqref{eq3.1}, we  have
 $$
   \Delta\widetilde{\phi}\nabla\phi=-(\widetilde{v}-\widetilde{w})\nabla(-\Delta)^{-1}(v-w).
$$
Applying Lemmas \ref{le4.2} and \ref{le4.3} to get that
\begin{align*}
  \frac{d}{dt}\|\Delta_{j}\widetilde{u}\|_{L^{p}}+\kappa2^{2j}\|\Delta_{j}\widetilde{u}\|_{L^{p}}
   \leq C2^{(1-\frac{3}{p})j}d_{j}Y'(t)e^{-KY(t)}\mathcal{E}(t)
   -KY'(t)\|\Delta_{j}\widetilde{u}\|_{L^{p}},
\end{align*}
which leads directly to
\begin{align}\label{eq4.13}
  \frac{d}{dt}\|\widetilde{u}\|_{\dot{B}^{-1+\frac{3}{p}}_{p,1}}+\kappa\|\widetilde{u}\|_{\dot{B}^{1+\frac{3}{p}}_{p,1}}
   \leq CY'(t)e^{-KY(t)}\mathcal{E}(t)
   -KY'(t)\|\widetilde{u}\|_{\dot{B}^{-1+\frac{3}{p}}_{p,1}}.
\end{align}

\textbf{The estimates of $v$ and $w$}
We only show the desired estimates for $\widetilde{v}$ due to $\widetilde{w}$ can be done analogously as $\widetilde{v}$.
 Applying the dyadic operator $\Delta_{j}$ to the third equation of \eqref{eq4.2}, then multiplying $|\Delta_{j}\widetilde{v}|^{q-2}\Delta_{j}\widetilde{v}$ and integrating over $\mathbb{R}^{3}$ (when
$q\in(1,2)$, we need to make some modification as that in
\cite{D01}), we see that
\begin{align}\label{eq4.14}
  \frac{1}{q}\frac{d}{dt}&\|\Delta_{j}\widetilde{v}\|_{L^{q}}^{q}
  -\int_{\mathbb{R}^{3}}\Delta\Delta_{j}\widetilde{v}|\Delta_{j}\widetilde{v}|^{q-2}\Delta_{j}\widetilde{v}dx=
  -\big(\Delta_{j}(u\cdot\nabla\widetilde{v})\big{|}|\Delta_{j}\widetilde{v}|^{q-2}\Delta_{j}\widetilde{v}\big)\nonumber\\
  &-\big(\Delta_{j}\nabla\cdot(\widetilde{v}\nabla\phi)\big{|}|\Delta_{j}\widetilde{v}|^{q-2}\Delta_{j}\widetilde{v}\big)
  -KY'(t)\|\Delta_{j}\widetilde{v}\|_{L^{q}}^{q}\nonumber\\
  &\leq \|\Delta_{j}(u\cdot\nabla\widetilde{v})\|_{L^{q}}\|\Delta_{j}\widetilde{v}\|_{L^{q}}^{q-1}
  +\|\Delta_{j}\nabla\cdot(\widetilde{v}\nabla\phi)\|_{L^{q}}\|\Delta_{j}\widetilde{v}\|_{L^{q}}^{q-1}-KY'(t)\|\Delta_{j}\widetilde{v}\|_{L^{q}}^{q}.
\end{align}
Thanks to \cite{D01,P00}, there exists a positive constant $\kappa$
so that
\begin{equation*}
  -\int_{\mathbb{R}^{3}}\Delta\Delta_{j}\widetilde{v}\cdot|\Delta_{j}\widetilde{v}|^{q-2}\Delta_{j}\widetilde{v}dx\geq
  \kappa2^{2j}\|\Delta_{j}\widetilde{v}\|_{L^{q}}^{q}.
\end{equation*}
Back to \eqref{eq4.14}, we obtain that
\begin{align*}
  \frac{d}{dt}\|\Delta_{j}\widetilde{v}\|_{L^{q}}+\kappa2^{2j}\|\Delta_{j}\widetilde{v}\|_{L^{q}}
   &\lesssim \|\Delta_{j}(u\cdot\nabla\widetilde{v})\|_{L^{q}}
  +\|\Delta_{j}\nabla\cdot(\widetilde{v}\nabla\phi)\|_{L^{q}}-KY'(t)\|\Delta_{j}\widetilde{v}\|_{L^{q}}.
\end{align*}
Lemmas \ref{le4.4} and \ref{le4.5} gives us to
\begin{align*}
  \frac{d}{dt}\|\Delta_{j}\widetilde{v}\|_{L^{q}}+\kappa2^{2j}\|\Delta_{j}\widetilde{v}\|_{L^{q}}
   &\lesssim 2^{(2-\frac{3}{q})j}d_{j}Y'(t)e^{-KY(t)}\mathcal{E}(t)-KY'(t)\|\Delta_{j}\widetilde{v}\|_{L^{q}},
\end{align*}
which implies directly that
\begin{align}\label{eq4.15}
  \frac{d}{dt}\|\widetilde{v}\|_{\dot{B}^{-2+\frac{3}{q}}_{q,1}}+\kappa\|\widetilde{v}\|_{\dot{B}^{\frac{3}{q}}_{q,1}}
   &\leq CY'(t)e^{-KY(t)}\mathcal{E}(t)-KY'(t)\|\widetilde{v}\|_{\dot{B}^{-2+\frac{3}{q}}_{q,1}}.
\end{align}
Similarly, for $w$,  we have
\begin{align}\label{eq4.16}
  \frac{d}{dt}\|\widetilde{w}\|_{\dot{B}^{-2+\frac{3}{q}}_{q,1}}+\kappa\|\widetilde{w}\|_{\dot{B}^{\frac{3}{q}}_{q,1}}
   &\leq  CY'(t)e^{-KY(t)}\mathcal{E}(t)-KY'(t)\|\widetilde{w}\|_{\dot{B}^{-2+\frac{3}{q}}_{q,1}}.
\end{align}

\textbf{Proof  of Proposition \ref{pro4.1}}
It is clear that from  \eqref{eq4.13}, \eqref{eq4.15}--\eqref{eq4.16}, there exists a constant $C$ such that
\begin{align*}
  \frac{d}{dt}(e^{-KY(t)}\mathcal{E}(t))+
  \kappa\big(\|\widetilde{u}(t)\|_{\dot{B}^{1+\frac{3}{p}}_{p,1}}+\|(\widetilde{v}(t),\widetilde{w}(t))\|_{\dot{B}^{\frac{3}{q}}_{q,1}}\big)
   &\leq (C-K)Y'(t)e^{-KY(t)}\mathcal{E}(t).
\end{align*}
By choosing $K$ sufficiently large such that $K>C$, we see that
\begin{align*}
  \frac{d}{dt}(e^{-KY(t)}\mathcal{E}(t))+
  \kappa\big(\|\widetilde{u}(t)\|_{\dot{B}^{1+\frac{3}{p}}_{p,1}}+\|(\widetilde{v}(t),\widetilde{w}(t))\|_{\dot{B}^{\frac{3}{q}}_{q,1}}\big)
   &\leq 0.
\end{align*}
We finish the proof of Proposition \ref{pro4.1}.

\subsection{Higher-order derivatives estimates}

Next we derive the higher-order spatial derivatives of the solutions to the system \eqref{eq3.1}--\eqref{eq3.2}.  Let $\ell$ be a real number and $1<r<\infty$.
Define
$$
\mathcal{F}(t):= \|u(t)\|_{\dot{B}^{\ell}_{r,1}}+\|(v(t),w(t))\|_{\dot{B}^{\ell-1}_{r,1}}.
$$
We obtain the following result.

\begin{proposition}\label{pro4.6}
Under the assumptions of Proposition \ref{pro4.1}, if we further assume that
$u_{0}\in \dot{B}^{\ell}_{r,1}(\mathbb{R}^{3})$, $v_{0}, w_{0}\in \dot{B}^{\ell-1}_{r,1}(\mathbb{R}^{3})$ with $1<r<\infty$, and
$$
  \frac{3}{p}+\ell>3\max\{0, \frac{1}{p}+\frac{1}{r}-1\}\ \ \text{and}\ \ \frac{3}{q}+\ell>3\max\{0, \frac{1}{q}+\frac{1}{r}-1\},
$$
then there exist two positive constants $\kappa$ and $K$ such that for all $t\geq0$,  the unique solution $(u,v,w)$ of the system \eqref{eq3.1}--\eqref{eq3.2} satisfies
\begin{align}\label{eq4.17}
  \frac{d}{dt}(e^{-KY(t)}\mathcal{F}(t))+ \frac{\kappa}{2} e^{-KY(t)}\big(\|u(t)\|_{\dot{B}^{\ell+2}_{r,1}}+\|(v(t),w(t))\|_{\dot{B}^{\ell+1}_{r,1}}\big)\leq0.
\end{align}
\end{proposition}
\begin{proof}
Applying the operator $\Delta_{j}\Lambda^{\ell}$ to the first equation of
\eqref{eq3.1},  and $\Delta_{j}\Lambda^{\ell-1}$ to the third equation of \eqref{eq3.1},
then  taking $L^{2}$ inner product with $|\Delta_{j}\Lambda^{\ell}u|^{r-2}\Delta_{j}\Lambda^{\ell}u$ to the first resultant, and  $|\Delta_{j}\Lambda^{\ell-1}v|^{r-2}\Delta_{j}\Lambda^{\ell-1}v$ to the second resultant (while for $1<r<2$,  we need to make some modification as that in
\cite{D01}),  we obtain  that
\begin{align*}
  \frac{1}{r}\frac{d}{dt}\|\Delta_{j}\Lambda^{\ell}u\|_{L^{r}}^{r}
  &-\big(\Delta\Delta_{j}\Lambda^{\ell}u\big{|}|\Delta_{j}\Lambda^{\ell}u|^{r-2}\Delta_{j}\Lambda^{\ell}u\big)=
  -\big(\Delta_{j}\Lambda^{\ell}(u\cdot\nabla u)\big{|}|\Delta_{j}\Lambda^{\ell}u|^{r-2}\Delta_{j}\Lambda^{\ell}u\big)\nonumber\\
  &+\big(\Delta_{j}\Lambda^{\ell}(\Delta\phi\nabla\phi)\big{|}|\Delta_{j}\Lambda^{\ell}u|^{r-2}\Delta_{j}\Lambda^{\ell}u\big)\nonumber\\
  &\leq \big(\|\Delta_{j}\Lambda^{\ell}(u\cdot\nabla u)\|_{L^{r}}
  +\|\Delta_{j}\Lambda^{\ell}(\Delta\phi\nabla\phi)\|_{L^{r}}\big)\|\Delta_{j}\Lambda^{\ell}u\|_{L^{r}}^{r-1},
\end{align*}
\begin{align*}
  \frac{1}{r}\frac{d}{dt}\|\Delta_{j}\Lambda^{\ell-1}v\|_{L^{r}}^{r}
  &-\big(\Delta\Delta_{j}\Lambda^{\ell-1}v\big{|}|\Delta_{j}\Lambda^{\ell-1}v|^{r-2}\Delta_{j}\Lambda^{\ell-1}v\big)\nonumber\\=&
  -\big(\Delta_{j}\Lambda^{\ell-1}(u\cdot\nabla v)\big{|}|\Delta_{j}\Lambda^{\ell-1}v|^{r-2}\Delta_{j}\Lambda^{\ell-1}v\big)\nonumber\\
  &-\big(\Delta_{j}\Lambda^{\ell-1}\nabla\cdot(v\nabla\phi)\big{|}|\Delta_{j}\Lambda^{\ell-1}v|^{r-2}\Delta_{j}\Lambda^{\ell-1}v\big)\nonumber\\
 \leq & \big(\|\Delta_{j}\Lambda^{\ell-1}(u\cdot\nabla v)\|_{L^{r}}
  +\|\Delta_{j}\Lambda^{\ell-1}\nabla\cdot(v\nabla\phi)\|_{L^{r}}\big)\|\Delta_{j}\Lambda^{\ell-1}v\|_{L^{r}}^{r-1},
\end{align*}
where we have used the fact
\begin{equation*}
  \int_{\mathbb{R}^{3}}\nabla\Delta_{j}\Lambda^{\ell}\Pi|\Delta_{j}\Lambda^{\ell}u|^{r-2}\Delta_{j}\Lambda^{\ell}udx=0,
\end{equation*}
which follows from the incompressibility condition $\nabla\cdot u=0$. Thanks again to \cite{D01,P00}, there exists a positive constant $\kappa$
so that
\begin{equation*}
  -\int_{\mathbb{R}^{3}}\Delta\Delta_{j}\Lambda^{\ell}u\cdot|\Delta_{j}\Lambda^{\ell}u|^{r-2}\Delta_{j}\Lambda^{\ell}udx\geq
  \kappa2^{2j}\|\Delta_{j}\Lambda^{\ell}u\|_{L^{r}}^{r},
\end{equation*}
\begin{equation*}
  -\int_{\mathbb{R}^{3}}\Delta\Delta_{j}\Lambda^{\ell-1}v\cdot|\Delta_{j}\Lambda^{\ell-1}v|^{r-2}\Delta_{j}\Lambda^{\ell-1}vdx\geq
  \kappa2^{2j}\|\Delta_{j}\Lambda^{\ell-1}v\|_{L^{r}}^{r}.
\end{equation*}
It follows that
\begin{align}\label{eq4.18}
  \frac{d}{dt}\|\Delta_{j}\Lambda^{\ell}u\|_{L^{r}}+\kappa2^{2j}\|\Delta_{j}\Lambda^{\ell}u\|_{L^{r}}
   \lesssim \|\Delta_{j}\Lambda^{\ell}(u\cdot\nabla u)\|_{L^{r}}
  +\|\Delta_{j}\Lambda^{\ell}(\Delta\phi\nabla\phi)\|_{L^{r}},
\end{align}
\begin{align}\label{eq4.19}
  \frac{d}{dt}\|\Delta_{j}\Lambda^{\ell-1}v\|_{L^{r}}+\kappa2^{2j}\|\Delta_{j}\Lambda^{\ell-1}v\|_{L^{r}}
   \lesssim \|\Delta_{j}\Lambda^{\ell-1}(u\cdot\nabla v)\|_{L^{r}}
  +\|\Delta_{j}\Lambda^{\ell-1}\nabla\cdot(v\nabla\phi)\|_{L^{r}}.
\end{align}
Taking $l^{1}$ norm to \eqref{eq4.18} and \eqref{eq4.19}, respectively,  and using Lemma \ref{le5.2},  we see that
\begin{align}\label{eq4.20}
  \frac{d}{dt}\|u\|_{\dot{B}^{\ell}_{r,1}}+\kappa\|u\|_{\dot{B}^{\ell+2}_{r,1}}
   &\lesssim \|u\cdot\nabla u\|_{\dot{B}^{\ell}_{r,1}}
  +\|\Delta\phi\nabla\phi\|_{\dot{B}^{\ell}_{r,1}},
\end{align}
\begin{align}\label{eq4.21}
  \frac{d}{dt}\|v\|_{\dot{B}^{\ell-1}_{r,1}}+\kappa\|v\|_{\dot{B}^{\ell+1}_{r,1}}
   &\lesssim \|u\cdot\nabla v\|_{\dot{B}^{\ell-1}_{r,1}}
  +\|\nabla\cdot(v\nabla\phi)\|_{\dot{B}^{\ell-1}_{r,1}}.
\end{align}
In order to finish the proof of Proposition \ref{pro4.6}, the case $\ell>0$ is simple because of  $\dot{B}^{\ell}_{r,1}(\mathbb{R}^{3})\cap L^{\infty}(\mathbb{R}^{3})$ is a Banach algebra, and  $\dot{B}^{\frac{3}{p}}_{p,1}(\mathbb{R}^{3})\hookrightarrow\dot{B}^{0}_{\infty,1}(\mathbb{R}^{3})\hookrightarrow L^{\infty}(\mathbb{R}^{3})$ for all $1\leq p\leq \infty$, we obtain by using Lemma \ref{le5.2} that
\begin{align*}
  \|u\cdot\nabla u\|_{\dot{B}^{\ell}_{r,1}}&\lesssim \|u\|_{\dot{B}^{\ell}_{r,1}}\|\nabla u\|_{L^{\infty}}+ \|\nabla u\|_{\dot{B}^{\ell}_{r,1}}\| u\|_{L^{\infty}}\\
  &\lesssim \|u\|_{\dot{B}^{\ell}_{r,1}}\|\nabla u\|_{\dot{B}^{\frac{3}{p}}_{p,1}}+\|u\|_{\dot{B}^{\ell+1}_{r,1}}\| u\|_{\dot{B}^{\frac{3}{p}}_{p,1}}\\
  &\lesssim\|u\|_{\dot{B}^{\ell}_{r,1}}\|u\|_{\dot{B}^{1+\frac{3}{p}}_{p,1}}+\|u\|_{\dot{B}^{\ell}_{r,1}}^{\frac{1}{2}}\|u\|_{\dot{B}^{\ell+2}_{r,1}}^{\frac{1}{2}}
  \| u\|_{\dot{B}^{-1+\frac{3}{p}}_{p,1}}^{\frac{1}{2}}\| u\|_{\dot{B}^{1+\frac{3}{p}}_{p,1}}^{\frac{1}{2}}\\
  &\leq \frac{\kappa}{2}\|u\|_{\dot{B}^{\ell+2}_{r,1}}+C\|u\|_{\dot{B}^{1+\frac{3}{p}}_{p,1}}\|u\|_{\dot{B}^{\ell}_{r,1}},
\end{align*}
where we have used the interpolation inequalities:
$$
  \| u\|_{\dot{B}^{\frac{3}{p}}_{p,1}}\lesssim\| u\|_{\dot{B}^{-1+\frac{3}{p}}_{p,1}}^{\frac{1}{2}}\| u\|_{\dot{B}^{1+\frac{3}{p}}_{p,1}}^{\frac{1}{2}}, \ \ \ \|u\|_{\dot{B}^{\ell+1}_{r,1}}\lesssim\|u\|_{\dot{B}^{\ell+2}_{r,1}}^{\frac{1}{2}}
  \|u\|_{\dot{B}^{\ell}_{r,1}}^{\frac{1}{2}}.
$$
Similarly, we have
\begin{align*}
  \|\Delta\phi\nabla\phi\|_{\dot{B}^{\ell}_{r,1}}&=\|(v-w)\nabla(-\Delta)^{-1}(w-v)\|_{\dot{B}^{\ell}_{r,1}}\\
   &\lesssim
   \|\nabla(-\Delta)^{-1}(w-v)\|_{\dot{B}^{\ell}_{r,1}} \|v-w\|_{L^{\infty}}+\|v-w\|_{\dot{B}^{\ell}_{r,1}} \|\nabla(-\Delta)^{-1}(w-v)\|_{L^{\infty}}\\
   &\lesssim
   \|(v,w)\|_{\dot{B}^{\ell-1}_{r,1}} \|(v,w)\|_{\dot{B}^{\frac{3}{q}}_{q,1}}+\|(v,w)\|_{\dot{B}^{\ell}_{r,1}} \|(v,w)\|_{\dot{B}^{-1+\frac{3}{q}}_{q,1}}\\
    &\lesssim
   \|(v,w)\|_{\dot{B}^{\ell-1}_{r,1}} \|(v,w)\|_{\dot{B}^{\frac{3}{q}}_{q,1}}+\|(v,w)\|_{\dot{B}^{\ell-1}_{r,1}}^{\frac{1}{2}}\|(v,w)\|_{\dot{B}^{\ell+1}_{r,1}}^{\frac{1}{2}} \|(v,w)\|_{\dot{B}^{-2+\frac{3}{q}}_{q,1}}^{\frac{1}{2}} \|(v,w)\|_{\dot{B}^{\frac{3}{q}}_{q,1}}^{\frac{1}{2}} \\
   &\leq \frac{\kappa}{6}\|(v,w)\|_{\dot{B}^{\ell+1}_{r,1}}+C\|(v,w)\|_{\dot{B}^{\frac{3}{q}}_{q,1}}\|(v,w)\|_{\dot{B}^{\ell-1}_{r,1}};
\end{align*}
\begin{align*}
   \|u\cdot\nabla v\|_{\dot{B}^{\ell-1}_{r,1}}&\approx\|uv\|_{\dot{B}^{\ell}_{r,1}}\lesssim
   \|u\|_{\dot{B}^{\ell}_{r,1}} \|v\|_{\dot{B}^{\frac{3}{q}}_{q,1}}+\|u\|_{\dot{B}^{\frac{3}{p}}_{p,1}} \|v\|_{\dot{B}^{\ell}_{r,1}}\\
   &\lesssim \|u\|_{\dot{B}^{\ell}_{r,1}} \|v\|_{\dot{B}^{\frac{3}{q}}_{q,1}}+\|u\|_{\dot{B}^{-1+\frac{3}{p}}_{p,1}}^{\frac{1}{2}}\|u\|_{\dot{B}^{1+\frac{3}{p}}_{p,1}}^{\frac{1}{2}} \|v\|_{\dot{B}^{\ell-1}_{r,1}}^{\frac{1}{2}}\|v\|_{\dot{B}^{\ell+1}_{r,1}}^{\frac{1}{2}}\\
   &\leq \frac{\kappa}{6}\|v\|_{\dot{B}^{\ell+1}_{r,1}}+C\|v\|_{\dot{B}^{\frac{3}{q}}_{q,1}}\|u\|_{\dot{B}^{\ell}_{r,1}} +C\|u\|_{\dot{B}^{1+\frac{3}{p}}_{p,1}}\|v\|_{\dot{B}^{\ell-1}_{r,1}};
\end{align*}
\begin{align*}
   \|\nabla\cdot(v\nabla\phi)\|_{\dot{B}^{\ell-1}_{r,1}}&\approx\|v\nabla\phi\|_{\dot{B}^{\ell}_{r,1}}
   =\|v\nabla(-\Delta)^{-1}(w-v)\|_{\dot{B}^{\ell}_{r,1}}\\
  &\lesssim\|(v,w)\|_{\dot{B}^{\ell-1}_{r,1}} \|v\|_{\dot{B}^{\frac{3}{q}}_{q,1}}+\|(v,w)\|_{\dot{B}^{-1+\frac{3}{q}}_{q,1}}\|v\|_{\dot{B}^{\ell}_{r,1}}\\
  &\lesssim\|(v,w)\|_{\dot{B}^{\ell-1}_{r,1}} \|v\|_{\dot{B}^{\frac{3}{q}}_{q,1}}+\|(v,w)\|_{\dot{B}^{-2+\frac{3}{q}}_{q,1}}^{\frac{1}{2}}
  \|(v,w)\|_{\dot{B}^{\frac{3}{q}}_{q,1}}^{\frac{1}{2}}\|v\|_{\dot{B}^{\ell-1}_{r,1}}^{\frac{1}{2}}\|v\|_{\dot{B}^{\ell+1}_{r,1}}^{\frac{1}{2}}\\
   &\lesssim \frac{\kappa}{6}\|v\|_{\dot{B}^{\ell+1}_{r,1}}+
   C\|(v,w)\|_{\dot{B}^{\frac{3}{q}}_{q,1}}\|(v,w)\|_{\dot{B}^{\ell-1}_{r,1}}.
\end{align*}
On the other hand, in the case $\ell\leq 0$,  recall that
$$
  \frac{3}{p}+\ell>3\max\{0, \frac{1}{p}+\frac{1}{r}-1\}\ \ \text{and}\ \ \frac{3}{q}+\ell>3\max\{0, \frac{1}{q}+\frac{1}{r}-1\}.
$$
Applying  Lemma \ref{le5.3} with $f=u$, $g=\nabla u$, $s_{1}=\frac{3}{p}$, $s_{2}=\ell$, $p_{1}=p$, $p_{2}=r$, we have
\begin{align*}
   \|u\cdot\nabla u\|_{\dot{B}^{\ell}_{r,1}}\lesssim
   \|u\|_{\dot{B}^{\ell}_{r,1}} \|\nabla u\|_{\dot{B}^{\frac{3}{p}}_{p,1}}\lesssim \|u\|_{\dot{B}^{1+\frac{3}{p}}_{p,1}}\|u\|_{\dot{B}^{\ell}_{r,1}};
\end{align*}
Applying Lemma \ref{le5.3} with $f=v-w$, $g=\nabla(-\Delta)^{-1}(w-v)$, $s_{1}=\frac{3}{q}$, $s_{2}=\ell$, $p_{1}=q$, $p_{2}=r$, we have
\begin{align*}
   \|\Delta\phi\nabla\phi\|_{\dot{B}^{\ell}_{r,1}}&=\|(v-w)\nabla(-\Delta)^{-1}(w-v)\|_{\dot{B}^{\ell}_{r,1}}\nonumber\\
   &\lesssim
   \|\nabla(-\Delta)^{-1}(w-v)\|_{\dot{B}^{\ell}_{r,1}} \|v-w\|_{\dot{B}^{\frac{3}{q}}_{q,1}}\lesssim \|(v,w)\|_{\dot{B}^{\frac{3}{q}}_{q,1}}\|(v,w)\|_{\dot{B}^{\ell-1}_{r,1}};
\end{align*}
Applying  Lemma \ref{le5.3} with $f=u$, $g=v$, $s_{1}=\frac{3}{q}$, $s_{2}=\ell$, $p_{1}=q$, $p_{2}=r$, we have
\begin{align*}
   \|u\cdot\nabla v\|_{\dot{B}^{\ell-1}_{r,1}}\approx\|uv\|_{\dot{B}^{\ell}_{r,1}}\lesssim
   \|v\|_{\dot{B}^{\frac{3}{q}}_{q,1}}\|u\|_{\dot{B}^{\ell}_{r,1}};
\end{align*}
Applying  Lemma \ref{le5.3} with $f=\nabla(-\Delta)^{-1}(w-v)$, $g=v$, $s_{1}=\frac{3}{q}$, $s_{2}=\ell$, $p_{1}=q$, $p_{2}=r$, we have
\begin{align*}
   \|\nabla\cdot(v\nabla\phi)\|_{\dot{B}^{\ell-1}_{r,1}}\approx\|v\nabla\phi\|_{\dot{B}^{\ell}_{r,1}}\lesssim
   \|\nabla(-\Delta)^{-1}(w-v)\|_{\dot{B}^{\ell}_{r,1}} \|v\|_{\dot{B}^{\frac{3}{q}}_{q,1}}\lesssim
   \|v\|_{\dot{B}^{\frac{3}{q}}_{q,1}}\|(v,w)\|_{\dot{B}^{\ell-1}_{r,1}}.
\end{align*}
Therefore, we conclude that
\begin{align}\label{eq4.22}
  \frac{d}{dt}\|u\|_{\dot{B}^{\ell}_{r,1}}+\frac{\kappa}{2}\|u\|_{\dot{B}^{\ell+2}_{r,1}}
   &\leq \frac{\kappa}{6}\|(v,w)\|_{\dot{B}^{\ell+1}_{r,1}}\nonumber\\
   &+C\big(\|u\|_{\dot{B}^{1+\frac{3}{p}}_{p,1}}+\|(v,w)\|_{\dot{B}^{\frac{3}{q}}_{q,1}}\big)\big(\|u\|_{\dot{B}^{\ell}_{r,1}}
  +\|(v,w)\|_{\dot{B}^{\ell-1}_{r,1}}\big);
\end{align}
\begin{align}\label{eq4.23}
  \frac{d}{dt}\|v\|_{\dot{B}^{\ell-1}_{r,1}}+\frac{2\kappa}{3}\|v\|_{\dot{B}^{\ell+1}_{r,1}}
   &\leq C\big(\|u\|_{\dot{B}^{1+\frac{3}{p}}_{p,1}}+\|(v,w)\|_{\dot{B}^{\frac{3}{q}}_{q,1}}\big)\big(\|u\|_{\dot{B}^{\ell}_{r,1}}
  +\|(v,w)\|_{\dot{B}^{\ell-1}_{r,1}}\big).
\end{align}
Similarly,
\begin{align}\label{eq4.24}
  \frac{d}{dt}\|w\|_{\dot{B}^{\ell-1}_{r,1}}+\frac{2\kappa}{3}\|w\|_{\dot{B}^{\ell+1}_{r,1}}
   &\leq C\big(\|u\|_{\dot{B}^{1+\frac{3}{p}}_{p,1}}+\|(v,w)\|_{\dot{B}^{\frac{3}{q}}_{q,1}}\big)\big(\|u\|_{\dot{B}^{\ell}_{r,1}}
  +\|(v,w)\|_{\dot{B}^{\ell-1}_{r,1}}\big).
\end{align}
By adding \eqref{eq4.22}--\eqref{eq4.24} together, we finally obtain that
\begin{align*}
  \frac{d}{dt}&\big(\|u\|_{\dot{B}^{\ell}_{r,1}}+\|(v,w)\|_{\dot{B}^{\ell-1}_{r,1}}\big)+\frac{\kappa}{2}\big(\|u\|_{\dot{B}^{\ell+2}_{r,1}}
  +\|(v,w)\|_{\dot{B}^{\ell+1}_{r,1}}\big)\nonumber\\
   &\leq C\big (\|u\|_{\dot{B}^{1+\frac{3}{p}}_{p,1}}+\|(v,w)\|_{\dot{B}^{\frac{3}{q}}_{q,1}}\big)\big(\|u\|_{\dot{B}^{\ell}_{r,1}}+\|(v,w)\|_{\dot{B}^{\ell-1}_{r,1}}\big).
\end{align*}
This yields \eqref{eq4.17} immediately.  We complete the proof of Proposition \ref{pro4.6}.
\end{proof}

\subsection{Proof of Theorem \ref{th3.1}}

Now we present the proof of Theorem \ref{th3.1}.
We first mention that Proposition \ref{pro4.6} implies \eqref{eq3.4} directly, so it suffices to prove \eqref{eq3.5}. For this purpose, for any $s>0$ such that
$$
  \frac{3}{p}-s>3\max\{0, \frac{1}{p}+\frac{1}{r}-1\}\ \ \text{and}\ \ \frac{3}{q}-s>3\max\{0, \frac{1}{q}+\frac{1}{r}-1\},
$$
by choosing $\ell=-s$ in Proposition \ref{pro4.6}, we see that for all $t\geq0$,
\begin{align}\label{eq4.25}
  \|u(t)\|_{\dot{B}^{-s}_{r,1}}+\|(v(t),w(t))\|_{\dot{B}^{-s-1}_{r,1}}
  \leq C\left(\|u_{0}\|_{\dot{B}^{-s}_{r,1}}+\|(v_{0},w_{0})\|_{\dot{B}^{-s-1}_{r,1}}\right)\leq C_{0}.
\end{align}
This particularly gives \eqref{eq3.5} with $\ell=-s$.  On the other hand, for any $\ell\in(-s,N]$,  by interpolation inequalities in Lemma \ref{le5.2}, we have for all $t\geq0$,
\begin{align*}
   \|u(t)\|_{\dot{B}^{\ell}_{r,1}}&\leq C \|u(t)\|_{\dot{B}^{-s}_{r,1}}^{\frac{2}{\ell+s+2}}\|u(t)\|_{\dot{B}^{\ell+2}_{r,1}}^{1-\frac{2}{\ell+s+2}},\nonumber\\
    \|(v(t),w(t))\|_{\dot{B}^{\ell-1}_{r,1}}&\leq C \|(v(t),w(t))\|_{\dot{B}^{-s-1}_{r,1}}^{\frac{2}{\ell+s+2}}\|(v(t),w(t))\|_{\dot{B}^{\ell+1}_{r,1}}^{1-\frac{2}{\ell+s+2}}.
\end{align*}
This together with \eqref{eq4.25}  implies that
\begin{align*}
   \|u(t)\|_{\dot{B}^{\ell+2}_{r,1}}&\geq C\|u(t)\|_{\dot{B}^{\ell}_{r,1}}^{1+\frac{2}{\ell+s}} \|u(t)\|_{\dot{B}^{-s}_{r,1}}^{-\frac{2}{\ell+s}}\geq C\|u(t)\|_{\dot{B}^{\ell}_{r,1}}^{1+\frac{2}{\ell+s}} ,\nonumber\\
    \|(v(t),w(t))\|_{\dot{B}^{\ell+1}_{r,1}}&\geq C\|(v(t),w(t))\|_{\dot{B}^{\ell-1}_{r,1}}^{1+\frac{2}{\ell+s}} \|(v(t),w(t))\|_{\dot{B}^{-s-1}_{r,1}}^{-\frac{2}{\ell+s}}\geq C\|(v(t),w(t))\|_{\dot{B}^{\ell-1}_{r,1}}^{1+\frac{2}{\ell+s}}.
\end{align*}
It follows that
\begin{align}\label{eq4.26}
   \|u(t)\|_{\dot{B}^{\ell+2}_{r,1}}+\|(v(t),w(t))\|_{\dot{B}^{\ell+1}_{r,1}}&\geq C(\|u(t)\|_{\dot{B}^{\ell}_{r,1}}+\|(v(t),w(t))\|_{\dot{B}^{\ell-1}_{r,1}})^{1+\frac{2}{\ell+s}}\nonumber\\
   &=C\mathcal{F}(t)^{1+\frac{2}{\ell+s}}.
\end{align}
Plugging \eqref{eq4.26}  into  \eqref{eq4.17}, we see that
\begin{equation*}
   \frac{d}{dt}(e^{-KY(t)}\mathcal{F}(t))+ C e^{-KY(t)}\mathcal{F}(t)^{1+\frac{2}{\ell+s}}\leq 0,
\end{equation*}
which combining the fact that the function $Y(t)$ is positive along time evolution yields that
\begin{equation}\label{eq4.27}
   \frac{d}{dt}(e^{-KY(t)}\mathcal{F}(t))+C (e^{-KY(t)}\mathcal{F}(t))^{1+\frac{2}{\ell+s}}\leq 0.
\end{equation}
Solving this differential inequality directly, we obtain
\begin{equation*}
   \mathcal{F}(t)\leq e^{KY(t)}\left(\mathcal{F}(0)^{-\frac{2}{\ell+s}}+\frac{2C t}{\ell+s}\right)^{-\frac{\ell+s}{2}}.
\end{equation*}
Note that  the function $Y(t)$ is bounded by the initial data in Proposition \ref{pro4.1}. Hence, we see that  for all $t\geq0$, there exists a constant $C_{0}$ such that
\begin{equation}\label{eq4.28}
   \|u(t)\|_{\dot{B}^{\ell}_{r,1}}+\|(v(t),w(t))\|_{\dot{B}^{\ell-1}_{r,1}}\leq C_{0}\left(1+t\right)^{-\frac{\ell+s}{2}}.
\end{equation}
We complete the proof of Theorem \ref{th3.1}, as desired.

\subsection{Proof of Theorem \ref{th3.2}}

Since $1<r\leq\min\{p, q\}$,  we infer from the imbedding results in Lemma \ref{le5.2} that
$$
     \dot{B}^{-s}_{r,1}(\mathbb{R}^{3})\hookrightarrow\dot{B}^{-s-3(\frac{1}{r}-\frac{1}{p})}_{p,1}(\mathbb{R}^{3}) \ \ \ \text{and}\ \ \
       \dot{B}^{-s-1}_{r,1}(\mathbb{R}^{3})\hookrightarrow\dot{B}^{-s-1-3(\frac{1}{r}-\frac{1}{q})}_{q,1}(\mathbb{R}^{3}),
$$
which together with \eqref{eq4.25} leads to for all $t\ge0$,
\begin{align}\label{eq4.29}
  \|u(t)\|_{\dot{B}^{-s-3(\frac{1}{r}-\frac{1}{p})}_{r,1}}+\|(v(t),w(t))\|_{\dot{B}^{-s-1-3(\frac{1}{r}-\frac{1}{q})}_{r,1}}
  \leq C_{0}.
\end{align}
 On the other hand, for any $s\ge\max\{0, 2-\frac{3}{r}\}$,  by interpolation inequalities in Lemma \ref{le5.2}, we have for all $t\geq0$,
\begin{align*}
    \|u(t)\|_{\dot{B}^{-1+\frac{3}{p}}_{p,1}}&\leq C \|u(t)\|_{\dot{B}^{-s-3(\frac{1}{r}-\frac{1}{p})}_{p,1}}^{\frac{2}{s+\frac{3}{r}+1}}\|u(t)\|_{\dot{B}^{1+\frac{3}{p}}_{p,1}}^{1-\frac{2}{s+\frac{3}{r}+1}},\nonumber\\
    \|(v(t),w(t))\|_{\dot{B}^{-2+\frac{3}{q}}_{q,1}}&\leq C \|(v(t),w(t))\|_{\dot{B}^{-s-1-3(\frac{1}{r}-\frac{1}{q})}_{q,1}}^{\frac{2}{s+\frac{3}{r}+1}}\|(v(t),w(t))\|_{\dot{B}^{\frac{3}{q}}_{q,1}}^{1-\frac{2}{s+\frac{3}{r}+1}}.
\end{align*}
This together with \eqref{eq4.29}  implies that
\begin{align*}
   \|u(t)\|_{\dot{B}^{1+\frac{3}{p}}_{p,1}}&\geq C\|u(t)\|_{\dot{B}^{-s-3(\frac{1}{r}-\frac{1}{p})}_{p,1}}^{\frac{2}{s+\frac{3}{r}-1}} \|u(t)\|_{\dot{B}^{-1+\frac{3}{p}}_{p,1}}^{1+\frac{2}{s+\frac{3}{r}-1}}
   \geq C\|u(t)\|_{\dot{B}^{-1+\frac{3}{p}}_{p,1}}^{1+\frac{2}{s+\frac{3}{r}-1}} ,\nonumber\\
    \|(v(t),w(t))\|_{\dot{B}^{\frac{3}{q}}_{q,1}}&\geq C\|(v(t),w(t))\|_{\dot{B}^{-s-1-3(\frac{1}{r}-\frac{1}{q})}_{q,1}}^{\frac{2}{s+\frac{3}{r}-1}}  \|(v(t),w(t))\|_{\dot{B}^{-2+\frac{3}{q}}_{q,1}}^{1+\frac{2}{s+\frac{3}{r}-1}} \geq C\|(v(t),w(t))\|_{\dot{B}^{-2+\frac{3}{q}}_{r,1}}^{1+\frac{2}{s+\frac{3}{r}-1}}.
\end{align*}
It follows that
\begin{align}\label{eq4.30}
   \|u(t)\|_{\dot{B}^{1+\frac{3}{p}}_{p,1}}+\|(v(t),w(t))\|_{\dot{B}^{\frac{3}{q}}_{q,1}}&\geq C(\|u(t)\|_{\dot{B}^{-1+\frac{3}{p}}_{p,1}}+\|(v(t),w(t))\|_{\dot{B}^{-2+\frac{3}{q}}_{q,1}})^{1+\frac{2}{s+\frac{3}{r}-1}}\nonumber\\
   &=C\mathcal{E}(t)^{1+\frac{2}{s+\frac{3}{r}-1}}.
\end{align}
Plugging \eqref{eq4.30}  into  \eqref{eq4.1}, by using the function $Y(t)$ is positive along time evolution, we obtain
\begin{equation}\label{eq4.31}
   \frac{d}{dt}(e^{-KY(t)}\mathcal{E}(t))+C (e^{-KY(t)}\mathcal{E}(t))^{1+\frac{2}{s+\frac{3}{r}-1}}\leq 0.
\end{equation}
Solving this differential inequality directly, we obtain
\begin{equation*}
   \mathcal{E}(t)\leq e^{KY(t)}\left(\mathcal{E}(0)^{-\frac{2}{s+\frac{3}{r}-1}}+\frac{2C t}{s+\frac{3}{r}-1}\right)^{-\frac{s+\frac{3}{r}-1}{2}}.
\end{equation*}
 Since $Y(t)$ is bounded by the initial data in Proposition \ref{pro4.1},  there exists a constant $C_{0}$ such that for all $t\geq0$,
\begin{equation}\label{eq4.32}
   \|u(t)\|_{\dot{B}^{-1+\frac{3}{p}}_{p,1}}+\|(v(t),w(t))\|_{\dot{B}^{-2+\frac{3}{q}}_{q,1}}\leq C_{0}\left(1+t\right)^{-\frac{s+\frac{3}{r}-1}{2}}.
\end{equation}
Notice that \eqref{eq4.32} gives in particular \eqref{eq3.6}  with $\ell=-1+\frac{3}{p}$, and
\eqref{eq3.7} with $\ell-1=-2+\frac{3}{q}$, respectively.
Finally,  for any  $\ell\in[-s-3(\frac{1}{r}-\frac{1}{p}), -1+\frac{3}{p})$, by using interpolation inequality in Lemma \ref{le5.2}, we see that
\begin{equation*}
     \|u(t)\|_{\dot{B}^{\ell}_{r,1}}\leq C\|u(t)\|_{\dot{B}^{-s-3(\frac{1}{r}-\frac{1}{p})}_{p,1}}^{\frac{\frac{3}{p}-1-\ell}{s+\frac{3}{r}-1}}
     \|u(t)\|_{\dot{B}^{-1+\frac{3}{p}}_{p,1}}^{\frac{\ell+s+3(\frac{1}{r}-\frac{1}{p})}{s+\frac{3}{r}-1}},
\end{equation*}
which combining \eqref{eq4.29} and \eqref{eq4.32} implies that
\begin{equation*}
     \|u(t)\|_{\dot{B}^{\ell}_{r,1}}\leq C_{0}(1+t)^{-(\frac{\ell+s}{2})-\frac{3}{2}(\frac{1}{r}-\frac{1}{p})}.
\end{equation*}
Similarly, for any  $\ell\in[-s-1-3(\frac{1}{r}-\frac{1}{q}), -2+\frac{3}{q})$,  there exists a constant $C_{0}$ such that for all $t\geq 0$,
\begin{equation*}
     \|(v(t), w(t))\|_{\dot{B}^{\ell-1}_{r,1}}\leq C\|(v(t),w(t))\|_{\dot{B}^{-s-1-3(\frac{1}{r}-\frac{1}{q})}_{q,1}}^{\frac{\frac{3}{q}-1-\ell}{s+\frac{3}{r}-1}}
     \|(v(t),w(t))\|_{\dot{B}^{-2+\frac{3}{q}}_{q,1}}^{\frac{\ell+s+3(\frac{1}{r}-\frac{1}{q})}{s+\frac{3}{r}-1}},
\end{equation*}
which combining \eqref{eq4.29} and \eqref{eq4.32} again leads to
\begin{equation*}
     \|(v(t), w(t))\|_{\dot{B}^{\ell-1}_{r,1}}\leq C_{0}(1+t)^{-(\frac{\ell+s}{2})-\frac{3}{2}(\frac{1}{r}-\frac{1}{q})}.
\end{equation*}
We complete the proof of Theorem \ref{th3.2}, as desired.
\section{Appendix}
 We first recall some crucial analytic tools used in the proofs of Theorems \ref{th3.1} and \ref{th3.2}, then give a sketched proof for global existence part in Theorem \ref{th3.1}.

\subsection{Useful lemmas}

\begin{lemma}\label{le5.1} {\em (\cite{BCD11}, \cite{D05})}
Let $\mathcal{B}$ be a ball, and $\mathcal{C}$  a ring in
$\mathbb{R}^{3}$. There exists a constant $C$ such that for any
positive real number $\lambda$, any nonnegative integer $k$ and any
couple of real numbers $(a,b)$ with $1\leq a\le b\leq \infty$, we
have
\begin{equation}\label{eq5.1}
   \operatorname{supp}\hat{f}\subset\lambda\mathcal{B}\ \ \Rightarrow\ \   \sup_{|\alpha|=k}\|\partial^{\alpha}f\|_{L^{b}}\leq
   C^{k+1}\lambda^{k+3(\frac{1}{a}-\frac{1}{b})}\|f\|_{L^{a}},
\end{equation}
\begin{equation}\label{eq5.2}
   \operatorname{supp}\hat{f}\subset\lambda\mathcal{C} \ \ \Rightarrow\ \   C^{-1-k}\lambda^{k}\|f\|_{L^{a}}\leq
   \sup_{|\alpha|=k}\|\partial^{\alpha}f\|_{L^{a}}\leq  C^{1+k}\lambda^{k}\|f\|_{L^{a}}.
\end{equation}
\end{lemma}

Let us now state some basic properties of Besov spaces (see \cite{BCD11}, \cite{D05}).

\begin{lemma}\label{le5.2}  {\em (\cite{BCD11}, \cite{D05})}
The following properties hold:
\begin{itemize}
\item [i)] Density: The set $C_{0}^{\infty}(\mathbb{R}^{3})$ is dense in $\dot{B}^{s}_{p,r}(\mathbb{R}^{3})$ if $|s|<\frac{3}{p}$ and $1\leq p,r<\infty$ or $s=\frac{3}{p}$ and $r=1$.

\item [ii)] Derivatives: There exists a universal constant $C$ such that
\begin{equation*}
    C^{-1}\|u\|_{\dot{B}^{s}_{p,r}}\leq \|\nabla u\|_{\dot{B}^{s-1}_{p,r}}\leq C\|u\|_{\dot{B}^{s}_{p,r}}.
\end{equation*}

\item [iii)] Fractional derivative: Let $\Lambda=\sqrt{-\Delta}$ and $\sigma\in\mathbb{R}$. Then the operator $\Lambda^{\sigma}$ is an isomorphism from $\dot{B}^{s}_{p,r}(\mathbb{R}^{3})$
to $\dot{B}^{s-\sigma}_{p,r}(\mathbb{R}^{3})$.

\item [iv)] Algebraic properties: For $s>0$, $\dot{B}^{s}_{p,r}(\mathbb{R}^{3})\cap L^{\infty}(\mathbb{R}^{3})$ is an algebra.  Moreover,  $\dot{B}^{\frac{3}{p}}_{p,1}(\mathbb{R}^{3})\hookrightarrow \dot{B}^{0}_{\infty,1}(\mathbb{R}^{3})\hookrightarrow L^{\infty}(\mathbb{R}^{3})$, and for any $f,g\in\dot{B}^{s}_{p,r}(\mathbb{R}^{3})\cap L^{\infty}(\mathbb{R}^{3})$, we have
$$
  \|fg\|_{\dot{B}^{s}_{p,r}}\leq  \|f\|_{\dot{B}^{s}_{p,r}}\|g\|_{L^{\infty}}+ \|g\|_{\dot{B}^{s}_{p,r}}\|f\|_{L^{\infty}}.
$$

\item [v)] Imbedding: For $1\leq p_{1}\leq p_{2}\leq \infty$ and $1\leq r_{1}\leq r_{2}\leq \infty$, we have the continuous imbedding  $\dot{B}^{s}_{p_{1},r_{1}}(\mathbb{R}^{3})\hookrightarrow \dot{B}^{s-3(\frac{1}{p_{1}}-\frac{1}{p_{2}})}_{p_{2},r_{2}}(\mathbb{R}^{3})$.

\item [vi)] Interpolation:  For  $s_{1},s_{2}\in \mathbb{R}$ such that $s_{1}<s_{2}$ and $\theta\in (0,1)$,  there exists a constant $C$
such that
\begin{align*}
    \|u\|_{\dot{B}^{s_{1}\theta+s_{2}(1-\theta)}_{p,r}}\leq C\|u\|_{\dot{B}^{s_{1}}_{p,r}}^{\theta}\|u\|_{\dot{B}^{s_{2}}_{p,r}}^{1-\theta}.
\end{align*}
\end{itemize}
\end{lemma}

\begin{lemma}\label{le5.3}  Let $1\leq p_{1}, p_{2}\leq\infty$, and $s_{1}\le \frac{3}{p_1}$,
$s_2\leq \min\{\frac{3}{p_{1}}, \frac{3}{p_2}\}$ with
$s_1+s_2>3\max(0,\frac{1}{p_1}+\frac{1}{p_2}-1)$. Assume that $f\in
\dot{B}^{s_1}_{p_1,1}(\mathbb{R}^{3})$,
$g\in\dot{B}^{s_2}_{p_2,1}(\mathbb{R}^{3})$. Then
$fg\in\dot{B}^{s_1+s_2-\frac{3}{p_1}}_{p_2,1}(\mathbb{R}^{3})$, and
there exists a positive constant $C$ such that
\begin{equation}\label{eq5.3}
   \|fg\|_{\dot{B}^{s_1+s_2-\frac{3}{p_1}}_{p_2,1}}\leq C\|f\|_{\dot{B}^{s_1}_{p_1,1}}\|g\|_{\dot{B}^{s_2}_{p_2,1}}.
\end{equation}
\end{lemma}
\begin{proof}
The ideas comes essentially from \cite{D05}. Thanks to Bony's paraproduct decomposition, we have
\begin{equation*}
  fg=T_{f}g+T_{g}f+R(f,g).
\end{equation*}
Applying Lemma \ref{le5.1} gives
\begin{align}\label{eq5.4}
   \|\Delta_{j}T_{f}g\|_{L^{p_{2}}}&\lesssim \sum_{|j'-j|\le4}\|S_{j'-1}f\|_{L^{\infty}}\|\Delta_{j'}g\|_{L^{p_{2}}}\nonumber\\
   &\lesssim \sum_{|j'-j|\le4}\sum_{k\leq j'-2}2^{(\frac{3}{p_{1}}-s_{1})k}2^{s_{1}k}\|\Delta_{k}f\|_{L^{p_{1}}}\|\Delta_{j'}g\|_{L^{p_{2}}}\nonumber\\
   &\lesssim 2^{(\frac{3}{p_{1}}-s_{1}-s_{2})j}d_{j}\|f\|_{\dot{B}^{s_1}_{p_1,1}}\|g\|_{\dot{B}^{s_2}_{p_2,1}}.
\end{align}
For the term $T_{g}f$, in the case that $1\leq p_{1}\leq p_{2}$, it follows from Lemma \ref{le5.1} that
\begin{align}\label{eq5.5}
   \|\Delta_{j}T_{g}f\|_{L^{p_{2}}}&\lesssim \sum_{|j'-j|\le4}2^{3(\frac{1}{p_{1}}-\frac{1}{p_{2}})j'}\|S_{j'-1}g\|_{L^{\infty}}\|\Delta_{j'}f\|_{L^{p_{1}}}\nonumber\\
   &\lesssim \sum_{|j'-j|\le4}2^{3(\frac{1}{p_{1}}-\frac{1}{p_{2}})j'}\sum_{k\leq j'-2}2^{(\frac{3}{p_{2}}-s_{2})k}2^{s_{2}k}\|\Delta_{k}g\|_{L^{p_{2}}}\|\Delta_{j'}f\|_{L^{p_{1}}}\nonumber\\
   &\lesssim 2^{(\frac{3}{p_{1}}-s_{1}-s_{2})j}d_{j}\|f\|_{\dot{B}^{s_1}_{p_1,1}}\|g\|_{\dot{B}^{s_2}_{p_2,1}},
\end{align}
while in the case that $ p_{2}<p_{1}$,  we have
\begin{align}\label{eq5.6}
   \|\Delta_{j}T_{g}f\|_{L^{p_{2}}}&\lesssim \sum_{|j'-j|\le4}\|S_{j'-1}g\|_{L^{\frac{p_{1}p_{2}}{p_{1}-p_{2}}}}\|\Delta_{j'}f\|_{L^{p_{1}}}\nonumber\\
   &\lesssim \sum_{|j'-j|\le4}\sum_{k\leq j'-2}2^{(\frac{3}{p_{1}}-s_{2})k}2^{s_{2}k}\|\Delta_{k}g\|_{L^{p_{2}}}\|\Delta_{j'}f\|_{L^{p_{1}}}\nonumber\\
   &\lesssim 2^{(\frac{3}{p_{1}}-s_{1}-s_{2})j}d_{j}\|f\|_{\dot{B}^{s_1}_{p_1,1}}\|g\|_{\dot{B}^{s_2}_{p_2,1}}.
\end{align}
Here estimates \eqref{eq5.4}, \eqref{eq5.5} and \eqref{eq5.6} are verified since $s_{1}\le \frac{3}{p_1}$,
$s_2\leq \min\{\frac{3}{p_{1}}, \frac{3}{p_2}\}$.  Finally, in the case that $\frac{1}{p_1}+\frac{1}{p_2}\leq 1$, using Lemma \ref{le5.1} again yields
that
\begin{align}\label{eq5.7}
   \|\Delta_{j}R(f,g)\|_{L^{p_{2}}}&\lesssim 2^{\frac{3j}{p_1}}\sum_{j'\geq j-N_{0}}\|\Delta_{j'}f\|_{L^{p_1}}\|\widetilde{\Delta}_{j'}g\|_{L^{p_{2}}}\nonumber\\
   &\lesssim 2^{\frac{3j}{p_1}}\sum_{j'\geq j-N_{0}}2^{-(s_1+s_2)j'}2^{s_1j'}\|\Delta_{j'}f\|_{L^{p_1}}2^{s_2j'}\|\widetilde{\Delta}_{j'}g\|_{L^{p_{2}}}\nonumber\\
   &\lesssim 2^{(\frac{3}{p_{1}}-s_{1}-s_{2})j}d_{j}\|f\|_{\dot{B}^{s_1}_{p_1,1}}\|g\|_{\dot{B}^{s_2}_{p_2,1}}.
\end{align}
If $\frac{1}{p_1}+\frac{1}{p_2}>1$, then we find $p_{2}'$ such that $\frac{1}{p_2}+\frac{1}{p_{2}'}=1$, and
\begin{align}\label{eq5.8}
   \|\Delta_{j}R(f,g)\|_{L^{p_{2}}}&\lesssim 2^{3(1-\frac{1}{p_2})j}\sum_{j'\geq j-N_{0}}\|\Delta_{j'}f\widetilde{\Delta}_{j'}g\|_{L^{1}}\nonumber\\
   &\lesssim 2^{3(1-\frac{1}{p_2})j}\sum_{j'\geq j-N_{0}}\|\Delta_{j'}f\|_{L^{p_{2}'}}\|\widetilde{\Delta}_{j'}g\|_{L^{p_{2}}}\nonumber\\
   &\lesssim 2^{3(1-\frac{1}{p_2})j}\sum_{j'\geq j-N_{0}}2^{3(\frac{1}{p_1}+\frac{1}{p_2}-1)j'}\|\Delta_{j'}f\|_{L^{p_{1}}}\|\widetilde{\Delta}_{j'}g\|_{L^{p_{2}}}\nonumber\\
   &\lesssim 2^{3(1-\frac{1}{p_2})j}\sum_{j'\geq j-N_{0}}2^{3(\frac{1}{p_1}+\frac{1}{p_2}-1)j'}2^{-(s_1+s_2)j'}2^{s_1j'}\|\Delta_{j'}f\|_{L^{p_1}}2^{s_2j'}\|\widetilde{\Delta}_{j'}g\|_{L^{p_{2}}}\nonumber\\
   &\lesssim 2^{(\frac{3}{p_{1}}-s_{1}-s_{2})j}d_{j}\|f\|_{\dot{B}^{s_1}_{p_1,1}}\|g\|_{\dot{B}^{s_2}_{p_2,1}}.
\end{align}
Here estimates \eqref{eq5.7} and \eqref{eq5.8} are verified since  $s_1+s_2>3\max(0,\frac{1}{p_1}+\frac{1}{p_2}-1)$. We complete the proof of Lemma \ref{le5.3}.
\end{proof}

\subsection{Global existence with small initial data}

In this section, we sketch the proof of global existence part in Theorem \ref{th3.1}. The approach is  similar to that of \cite{ZZL15}. The only difficulty lies in estimations of the nonlinear terms $\Delta\phi\nabla\phi$ and $u\cdot\nabla v$, which if $1\leq q\leq p$,  we have proved the desired bilinear estimates in \cite{ZZL15}. In the case of $p<q$, the proof comes essentially from the approach used in Lemma \ref{le4.2} to estimate $\Delta\phi\nabla\phi$,  and Lemma \ref{le4.3} to estimate $u\cdot\nabla v$, where we need only to deal with the time variable by the general principle that
the time exponent behaves according to the H\"{o}lder
inequality, and an additional condition $\frac{1}{q}-\frac{1}{p}> -\min\{\frac{1}{3}, \frac{1}{2p}\}$ is needed.
More precisely, since
$$
    \Delta\phi\nabla\phi=-(v-w)\nabla(-\Delta)^{-1}(v-w),
$$
we can show that
\begin{align*}
   \|v\nabla(-\Delta)^{-1}w\!+\!w\nabla(-\Delta)^{-1}v\|_{L^{1}_{t}(\dot{B}^{-1+\frac{3}{p}}_{p,1})}\lesssim \|v\|_{\mathcal{L}^{\infty}_{t}(\dot{B}^{-2+\frac{3}{q}}_{q,1})}
   \|w\|_{L^{1}_{t}(\dot{B}^{\frac{3}{q}}_{q,1})}\!+\!\|w\|_{\mathcal{L}^{\infty}_{t}(\dot{B}^{-2+\frac{3}{q}}_{q,1})}
   \|v\|_{L^{1}_{t}(\dot{B}^{\frac{3}{q}}_{q,1})}
\end{align*}
and
\begin{equation*}
   \|u\cdot\nabla
    v\|_{L^{1}_{t}(\dot{B}^{-2+\frac{3}{q}}_{q,1})}\lesssim\|u\|_{\mathcal{L}^{\infty}_{t}(\dot{B}^{-1+\frac{3}{p}}_{p,1})}
   \|v\|_{L^{1}_{t}(\dot{B}^{\frac{3}{q}}_{q,1})}+\|u\|_{L^{1}_{t}(\dot{B}^{1+\frac{3}{p}}_{p,1})}
   \|v\|_{\mathcal{L}^{\infty}_{t}(\dot{B}^{-2+\frac{3}{q}}_{q,1})}.
\end{equation*}
Based on these two desired bilinear estimates, we can follow the approach used in \cite{ZZL15} to prove that if $\|(u_0, v_0, w_0)\|_{\dot{B}^{-1+\frac{3}{p}}_{p,1}\times(\dot{B}^{-2+\frac{3}{q}}_{q,1})^{2}}$ is sufficiently small, then the system
\eqref{eq3.1}--\eqref{eq3.2} admits a unique global solution. We complete the proof, as desired.

\medskip
\medskip

\noindent\textbf{Acknowledgments.} J. Zhao is partially supported by the National Natural Science Foundation of China (11371294), the Fundamental Research Funds for the Central Universities (2014YB031) and the Fundamental Research Project of Natural Science in Shaanxi Province--Young Talent Project (2015JQ1004). Q. Liu is partially supported by the National Natural Science Foundation of China (11326155,
11401202),  the Scientific Research Fund of Hunan Provincial
Education Department (14B117)  and the China Postdoctoral Science Foundation (2015M570053).

\end{document}